\documentclass[10pt, oldfontcommands,  6x9]{pupbook}
\usepackage{amsmath, amscd, amssymb, amsfonts, xypic}
\usepackage{mathrsfs}
\usepackage{enumerate}
\usepackage{times}
\usepackage{appendix}
\usepackage[all]{xy}
\usepackage[english]{babel}
\usepackage[shortlabels]{enumitem}

\usepackage{makeidx}
\makeindex
\usepackage{graphicx}

    \newtheoremstyle{plain}%
        {8pt plus2pt minus4pt}%
        {8pt plus2pt minus4pt}%
        {\itshape}%
        {}%
        {\bfseries\scshape}%
        {}%
        {1em}%
        {}%

    \newtheoremstyle{upright}%
        {8pt plus2pt minus4pt}%
        {8pt plus2pt minus4pt}%
        {\upshape}%
        {}%
        {\bfseries\scshape}%
        {}%
        {1em}%
        {}%
\numberwithin{equation}{section}
\numberwithin{theorem}{section}
\theoremstyle{upright}

\newtheorem{rem}[theorem]{Remark}

\newtheorem{exa}[theorem]{Example}

\newtheorem{??}[theorem]{Question}

\theoremstyle{plain}
\newtheorem{prop}[theorem]{Proposition}
\newtheorem{cor}[theorem]{Corollary}

\newtheorem{defi}[theorem]{Definition}

\newtheorem{lemma-definition}[theorem]{Lemma and Definition}

\numberwithin{equation}{section}

\newcommand{\an}{\mathrm{an}}

\newcommand{\id}{\operatorname{id}}

\newcommand{\baseRing}[1]{\ensuremath{\mathbb{#1}}}
\newcommand{\Z}{\baseRing{Z}}
\newcommand{\R}{\baseRing{R}}
\newcommand{\C}{\baseRing{C}}

\newcommand{\bp}{\begin{proof}}
\newcommand{\ep}{\end{proof}}
\newcommand{\beq}{\begin{equation}}
\newcommand{\eeq}{\end{equation}}

\newcommand{\beqs}{\begin{equation*}}
\newcommand{\eeqs}{\end{equation*}}

\newcommand{\beas}{\begin{eqnarray*}}
\newcommand{\eeas}{\end{eqnarray*}}

\newcommand{\btheorem}{\begin{theorem}}
\newcommand{\etheorem}{\end{theorem}}

\newcommand{\bl}{\begin{lemma}}
\newcommand{\el}{\end{lemma}}

\newcommand{\benum}{\begin{enumerate}}
\newcommand{\eenum}{\end{enumerate}}

\newcommand{\cinf}{C^\infty}

\renewcommand{\tt}{\theta_2}

\usepackage{pstricks}
\usepackage{lineno}
\newcommand{\Alpha}{{\text{A}}}
\newcommand{\ga}{\alpha}
\newcommand{\gb}{\beta}
\newcommand{\gc}{\gamma}

\newcommand{\cala}{\mathcal{A}}
\newcommand{\calb}{\mathcal{B}}
\newcommand{\calc}{\mathcal{C}}

\newcommand{\cale}{\mathcal{E}}
\newcommand{\calf}{\mathcal{F}}
\newcommand{\calg}{\mathcal{G}}
\newcommand{\calh}{\mathcal{H}}

\newcommand{\call}{\mathcal{L}}

\newcommand{\calo}{\mathcal{O}}
\newcommand{\calq}{\mathcal{Q}}
\newcommand{\calr}{\mathcal{R}}
\newcommand{\cals}{\mathcal{S}}

\newcommand{\calz}{\mathcal{Z}}

\newcommand{\fraku}{\mathfrak{U}}

\newcommand{\Xand}{X_{\mathrm{an}}}
\def\suban{_{\mathrm{an}}}
\def\alg{_{\mathrm{alg}}}
\newcommand{\ap}{\ga_0 \ldots \ga_p}
\newcommand{\apone}{\ga_0 \ldots \ga_{p+1}}
\newcommand{\bu}{^{\bullet}}
\newcommand{\comp}{\mathrel{\scriptstyle\circ}}
\newcommand{\dpd}[2]{\dfrac{\partial #1}{\partial #2}}
\newcommand{\hh}{\mathbb{H}}
\newcommand{\im}{\operatorname{Im}}
\newcommand{\supp}{\operatorname{supp}}
\newcommand{\term}[1]{\textbf{\textit{#1}}}

\newtheorem{exer}[theorem]{Exercise}
\def \pf{\begin{proof}}
\def \epf{\end{proof}}

\def\bfig{\vcenter\bgroup}
\def\efig{\egroup}

\theoremstyle{upright}

\font\tenmsb=msbm10
\font\sevenmsb=msbm7
\font\fivemsb=msbm5

\newfam\msbfam
\textfont\msbfam=\tenmsb
\scriptfont\msbfam=\sevenmsb
\scriptscriptfont\msbfam=\fivemsb

\font\teneufm=eufm10
\font\seveneufm=eufm7
\font\fiveeufm=eufm5
\newfam\eufmfam
\textfont\eufmfam=\teneufm
\scriptfont\eufmfam=\seveneufm
\scriptscriptfont\eufmfam=\fiveeufm

\usepackage{units}

\begin{document}
\setcounter{page}{69}
\setcounter{chapter}{1}

\chapter[The Algebraic de Rham Theorem by F.~El Zein and L.~Tu]{From Sheaf Cohomology to
the Algebraic de Rham Theorem\\
by Fouad El Zein and Loring W.~Tu}\label{ch:elzein_tu}

\bigskip
\bigskip


\section*{Introduction }
\addcontentsline{toc}{section}{Introduction}
The concepts of homology and cohomology trace their origin
to the work of Poincar\'e in the late nineteenth century.
They attach to a topological space algebraic structures such as groups or rings that are topological invariants of the space.
There are actually many different theories, for example, simplicial, singular, and de Rham theories.
In 1931, Georges de Rham proved a conjecture of Poincar\'e
on a relationship between cycles and smooth differential forms,
which establishes for a smooth manifold an isomorphism
between singular cohomology with real coefficients and de Rham
cohomology.

More precisely, by integrating smooth forms over singular chains
on a smooth manifold $M$, one obtains a linear map
\[
\cala^k(M) \to S^k(M, \R)
\]
from the vector space $\cala^k(M)$ of smooth $k$-forms on $M$
to the vector space $S^k(M, \R)$ of real singular $k$-cochains on $M$.
The theorem of de Rham\index{de Rham theorem} asserts that this linear map induces an
isomorphism
\[
H_{\text{dR}}^*(M) \overset{\sim}{\to} H^*(M,\R)
\]
between the de Rham cohomology $H_{\text{dR}}^*(M)$ and the singular
cohomology $H^*(M,\R)$,
under which the wedge product of classes of closed smooth differential forms corresponds
to the cup product of classes of cocycles.
Using complex coefficients, there is similarly an isomorphism
\[
h^*\big(\cala\bu(M,\C)\big) \overset{\sim}{\to} H^*(M, \C),
\]
where $h^*\big(\cala\bu(M,\C)\big)$ denotes the cohomology of the
complex $\cala\bu(M,\C)$ of smooth $\C$-valued forms on $M$.

By an algebraic variety,\index{algebraic variety} we will mean
a reduced separated scheme of finite type over an
algebraically closed field \cite[Vol.~2, Ch.~VI, Sec.~1.1, p.~49]{shafarevich}.
In fact, the field throughout the article will be the field of complex
numbers.
For those not familiar with the language of schemes, there is
no harm in taking an algebraic variety to be a quasi-projective
variety; the proofs of the algebraic de Rham theorem
 are exactly the same in the two cases.

Let $X$ be a smooth complex algebraic variety with the Zariski
topology.
A \term{regular} function\index{regular function} on an open set $U\subset X$
is a rational function that is defined at every point of $U$.
A differential $k$-form on $X$ is \term{algebraic}\index{algebraic differential form} if locally it can be
written as $\sum f_I\, dg_{i_1} \wedge \cdots \wedge dg_{i_k}$
for some regular functions $f_I$, $g_{i_j}$.
With the complex topology, the underlying set of the smooth variety
$X$ becomes a complex manifold $\Xand$.
By de Rham's theorem, the singular
cohomology $H^*(\Xand, \C)$ can be computed from
the complex of smooth $\C$-valued differential forms
on $\Xand$.
Grothendieck's algebraic de Rham theorem\index{Grothendieck's algebraic de Rham theorem}%
\index{algebraic de Rham theorem}\index{de Rham theorem!algebraic} asserts
that the singular cohomology $H^*(\Xand, \C)$ can in fact be computed from
the complex $\Omega_{\alg}^{\bullet}$ of sheaves
of algebraic differential forms on $X$.
Since algebraic
de Rham cohomology can be defined over any field,
Grothendieck's theorem lies at the foundation of
Deligne's theory of absolute Hodge classes
(see Chapter~\ref{ch:charles_schnell} in this volume).

In spite of its beauty and importance, there does not seem to be an accessible
account of Grothendieck's \hyphenation{Grothen-dieck} algebraic
de Rham theorem in the literature.  Grothendieck's paper \cite{grothendieck66},
invoking higher direct images of sheaves and a theorem of
Grauert--Remmert, is quite difficult to read.
An impetus for our work is to give an elementary proof
of Grothendieck's theorem, elementary in the sense that
we use only tools from standard textbooks as well as
some results from Serre's groundbreaking FAC and GAGA papers
(\cite{serre55} and \cite{serre56}).

This article is in two parts.
In Part I, comprising Sections 1 through 6, we prove Grothendieck's algebraic
de Rham theorem more or less from scratch for a
smooth complex projective variety $X$, namely, that there is
an isomorphism
\[
H^*(\Xand,\C) \simeq
\hh^*(X, \Omega_{\alg}^{\bullet})
\]
between the complex singular cohomology of $\Xand$
and the hypercohomology of the complex $\Omega_{\alg}^{\bullet}$
of sheaves of algebraic differential forms on $X$.
The proof, relying mainly on Serre's GAGA principle
and the technique of hypercohomology,
necessitates a discussion of sheaf cohomology, coherent sheaves,
and hypercohomology, and so another goal is to
give an introduction to these topics.
While Grothendieck's theorem is valid as a ring isomorphism,
to keep the account simple, we prove only a vector space isomorphism.
In fact, we do not even discuss multiplicative structures on hypercohomology.
In Part II, comprising Sections 7 through 10,
we develop more machinery, mainly
the \v{C}ech cohomology of a sheaf and
the \v{C}ech cohomology of a complex
of sheaves, as tools for computing hypercohomology.
We prove that the general case of Grothendieck's theorem is equivalent to the
affine case, and then prove the affine case.

The reason for the two-part structure of our article is the sheer
amount of background needed to prove Grothendieck's
algebraic de Rham theorem in general.
It seems desirable to treat the simpler case of a smooth projective
variety first, so that the reader can see a major landmark before
being submerged in yet more machinery.
In fact, the projective case is not necessary to the proof of
the general case, although the tools developed, such as
sheaf cohomology and hypercohomology, are indispensable to the
general proof.
A reader who is already familiar with these tools can go directly
to Part II.

Of the many ways to define sheaf cohomology, for example
as \v{C}ech cohomology, as the cohomology of global sections
of a certain resolution, or as an example of a right-derived
functor in an abelian category, each has its own merit.
We have settled on Godement's approach
using his
canonical resolution \cite[Sec.~4.3, p.~167]{godement}.
It has the advantage of being
the most direct.
Moreover, its extension to the hypercohomology of a complex
of sheaves gives at once the $E_2$ terms of the standard spectral
sequences converging to the hypercohomology.

What follows is a more detailed description
of each section.  In Part I, we recall in Section~1 some of the
properties of sheaves.
In Section~2, sheaf cohomology is defined as the cohomology
of the complex of global sections of Godement's
canonical resolution.
In Section~3, the cohomology of a sheaf is generalized to the
hypercohomology of a complex of sheaves.
Section~4 defines coherent analytic and algebraic sheaves
and summarizes Serre's GAGA principle
for a smooth complex projective variety.
Section~5 proves the holomorphic Poincar\'e lemma
and the analytic de Rham theorem for any complex manifold,
and Section~6 proves the algebraic
de Rham theorem for a smooth complex projective variety.

In Part II, we develop in Sections 7 and 8 the \v{C}ech cohomology of
a sheaf and of a complex of sheaves.
Section~9 reduces the algebraic de Rham theorem for an algebraic
variety to a theorem about affine varieties.
Finally, in Section~10 we treat the
affine case.


We are indebted to George Leger for his feedback and
to Jeffrey D.~Carlson for helpful discussions and detailed comments on
the many drafts of the article.
Loring Tu is also grateful to the Tufts University Faculty
Research Award Committee for a New Directions in Research Award and to
the National Center for Theoretical Sciences Mathematics Division (Taipei Office) in Taiwan for hosting him
during part of the preparation of this manuscript.

\noindent
\section*{Part I.~Sheaf Cohomology, Hypercohomology,
and the Projective Case}
\addcontentsline{toc}{section}{\textbf{Part I.~Sheaf Cohomology, Hypercohomology, and the Projective Case}}

\section{Sheaves}

We assume a basic knowledge of sheaves as
in \cite[Ch.~II, Sec.~1, pp.~60--69]{hartshorne}.

\subsection{The \'Etal\'e Space of a Presheaf} \label{ss:associated}

Associated to a presheaf $\calf$ on a topological space $X$ is another
topological space $E_{\calf}$, called the \term{\'etal\'e space}
of $\calf$.%
\index{etale space@\'etal\'e space}
Since the \'etal\'e space is needed in the construction of Godement's
canonical resolution of a sheaf, we give a brief discussion here.
As a set, the \'etal\'e space $E_{\calf}$ is the disjoint union
$\coprod_{p\in X} \calf_p$ of all the stalks of $\calf$.
There is a natural projection map
$\pi\colon E_{\calf} \to X$ that maps $\calf_p$ to $p$.
A \term{section}\index{section!of an \'etal\'e space} of the \'etal\'e space $\pi\colon E_{\calf} \to X$
over $U \subset X$ is a map $s\colon U \to E_{\calf}$ such that
$\pi \comp s = \id_U$, the identity map on $U$.
For any open set $U \subset X$, element $s\in \calf(U)$, and point $p
\in U$, let $s_p \in \calf_p$ be the germ of $s$ at $p$.
Then the element $s\in \calf(U)$ defines a section
$\tilde{s}$ of the \'etal\'e
space over $U$,
\begin{align*}
\tilde{s}\colon U &\to E_{\calf},\\
p &\mapsto s_p \in \calf_p.
\end{align*}
The collection
\[
\{ \tilde{s}(U) \mid U \text{ open in } X, \ s \in \calf(U) \}
\]
of subsets of $E_{\calf}$ satisfies the conditions to be a basis
for a topology on $E_{\calf}$.
With this topology, the \'etal\'e space $E_{\calf}$ becomes
a topological space.
By construction, the topological space $E_{\calf}$
is locally homeomorphic to $X$.
For any element $s\in \calf(U)$, the function $\tilde{s}\colon U \to E_{\calf}$
is a continuous section of $E_{\calf}$.
A section $t$ of the \'etal\'e space $E_{\calf}$ is continuous if
and only if every point $p \in X$ has a neighborhood $U$ such that $t
= \tilde{s}$ on $U$ for some $s \in \calf(U)$.

Let $\calf^+$ be the presheaf that associates to each
open subset $U \subset X$ the abelian group
\[
\calf^+(U) := \{ \text{continuous sections } t\colon U \to
E_{\calf} \}.
\]
Under pointwise addition of sections,
the presheaf $\calf^+$ is easily seen to be a sheaf, called the
\term{sheafification} or the \term{associated sheaf} of the presheaf $\calf$.%
\index{sheafification}\index{associated sheaf}
There is an obvious presheaf morphism $\theta\colon \calf \to
\calf^+$ that sends a section $s\in \calf(U)$ to the section
$\tilde{s} \in \calf^+(U)$.

\begin{exa}
For each open set $U$ in a topological space $X$,
let $\calf(U)$ be the group of all
\emph{constant} real-valued functions on $U$.
At each point $p \in X$, the stalk $\calf_p$ is $\R$.
The \'etal\'e space $E_{\calf}$ is thus $X \times \R$,
but not with its usual topology.
A basis for $E_{\calf}$ consists of open sets of the form
$U \times \{ r\}$ for an open set $U \subset X$ and a
number $r \in \R$.
Thus, the topology on $E_{\calf}= X \times \R$ is the
product topology of the given topology on $X$ and the
discrete topology on $\R$.
The sheafification $\calf^{+}$ is the sheaf $\underline{\R}$
of \emph{locally constant} real-valued functions.%
\index{presheaf of constant functions}\index{presheaf of constant functions!sheafification}
\end{exa}

\begin{exer} \label{1exer:sheafification}
Prove that if $\calf$ is a sheaf,
then $\calf \simeq \calf^+$.
(\textit{Hint}: the two sheaf axioms say precisely that for
every open set $U$, the map $\calf(U) \to \calf^+(U)$ is
one-to-one and onto.)
\end{exer}

\subsection{Exact Sequences of Sheaves} \label{s:exact}

From now on, we will consider only sheaves of abelian groups.
A sequence of morphisms of sheaves of abelian groups
\[
\cdots \longrightarrow \calf^1 \overset{d_1}{\longrightarrow}
\calf^2 \overset{d_2}{\longrightarrow}  \calf^3
\overset{d_3}{\longrightarrow} \cdots
\]
on a topological space $X$ is said to be \term{exact}\index{exact
sequence!of sheaves} at $\calf^k$ if $\im d_{k-1} = \ker d_k$;
the sequence is said to be \term{exact} if
it is exact at every $\calf^k$.
The exactness of a sequence of morphisms of sheaves
on $X$ is equivalent to the exactness of the sequence of stalk
maps at every point $p \in X$ (see \cite[Exer.~1.2, p.~66]{hartshorne}). An exact sequence of sheaves of the
form
\begin{equation} \label{e:short1}
0 \to \cale \to \calf \to \calg \to 0
\end{equation}
is said to be a \term{short exact sequence}.\index{short exact sequence!of
sheaves}

It is not too difficult to show that the exactness of the sheaf
sequence \eqref{e:short1} over a topological space $X$
implies the exactness of the
sequence of sections
\begin{equation} \label{e:left}
0 \to \cale(U) \to \calf(U) \to \calg(U)
\end{equation}
for every open set $U \subset X$,
but that the last map $\calf(U) \to \calg(U)$ need not
be surjective.
In fact, as we will see in Theorem~\ref{t:long},
the cohomology $H^1(U, \cale)$ is a measure of the
nonsurjectivity of the map $\calf(U) \to
\calg(U)$ of sections.

Fix an open subset $U$ of a topological space $X$.
To every sheaf $\calf$ of abelian groups on $X$, we can associate the
abelian group $\Gamma(U,\calf) := \calf(U)$
of sections over $U$
and to every sheaf map $\varphi\colon \calf \to \calg$,
the group homomorphism $\varphi_U\colon \Gamma(U,\calf) \to
\Gamma(U,\calg)$.
This makes $\Gamma(U,\ )$ a functor from sheaves
of abelian groups on $X$ to abelian
groups.

%

A functor $F$ from the category of sheaves of
abelian groups on $X$ to the
category of abelian groups is said to be \term{exact}%
\index{exact functor}\index{functor!exact}
if it maps a short exact sequence of sheaves
\[
0 \to \cale \to \calf \to \calg \to 0
\]
to a short exact sequence of abelian groups
\[
0 \to F(\cale) \to F(\calf) \to F(\calg) \to 0.
\]
If instead one has only the exactness of
\begin{equation} \label{e:leftexact}
0 \to F(\cale) \to F(\calf) \to F(\calg),
\end{equation}
then $F$ is said to be a \term{left-exact functor}.\index{left-exact functor}
The sections functor $\Gamma(U, \ )$ is left exact
but not exact.  (By Proposition~\ref{2p:sections} and
Theorem~\ref{t:long}, the next term in the exact
sequence \eqref{e:leftexact} is the first
cohomology group $H^1(U,\cale)$.)

\subsection{Resolutions}

Recall that $\underline{\R}$ is the sheaf of locally constant functions
with values in $\R$ and $\cala^k$ is the sheaf of smooth
$k$-forms on a manifold $M$.
The exterior derivative $d\colon \cala^k(U) \to \cala^{k+1}(U)$,
as $U$ ranges over all open sets in $M$,
defines a morphism of sheaves $d\colon \cala^k \to \cala^{k+1}$.

\begin{prop}
On any manifold $M$ of dimension $n$, the sequence of
sheaves
\begin{equation} \label{6e:dr}
0 \to \underline{\R} \to \cala^0 \overset{d}{\to} \cala^1
\overset{d}{\to} \cdots \overset{d}{\to} \cala^n \to 0
\end{equation}
is exact.
\end{prop}

\pf Exactness at $\cala^0$ is equivalent to the exactness of the
sequence of stalk maps $\underline{\R}_p \to \cala_p^0
\overset{d}{\to} \cala_p^1$ for all $p \in M$. Fix a point $p \in
M$. Suppose $[f] \in \cala_p^0$ is the germ of a $\cinf$ function
$f\colon U \to \R$, where $U$ is a neighborhood of $p$,
such that $d [f] = [0]$ in $\cala_p^1$. Then there is a
neighborhood $V \subset U$ of $p$ on which $df \equiv 0$. Hence,
$f$ is locally constant on $V$ and $[f] \in \underline{\R}_p$.
Conversely, if $[f] \in \underline{\R}_p$, then $d [f] =0$. This
proves the exactness of the sequence \eqref{6e:dr} at $\cala^0$.

Next, suppose $[\omega] \in \cala_p^k$ is the germ of a smooth
$k$-form $\omega$ on some neighborhood of $p$ such that $d [\omega] = 0 \in
\cala_p^{k+1}$. This means there is a neighborhood $V$ of $p$ on
which $d\omega \equiv 0$. By making $V$ smaller, we may assume
that $V$ is contractible.
By the Poincar\'e lemma \cite[Cor.~4.1.1, p.~35]{bott--tu},
$\omega$ is
exact on $V$, say $\omega = d\tau$ for some $\tau \in
\cala^{k-1}(V)$. Hence, $[\omega] = d [\tau]$ in $\cala_p^k$.
This proves the exactness of the sequence \eqref{6e:dr}
at $\cala^k$ for $k > 0$.  \epf

In general, an exact sequence of sheaves
\[
0 \to \cala \to \calf^0 \to \calf^1 \to \calf^2 \to \cdots
\]
on a topological space $X$ is called a
\term{resolution}\index{resolution} of the sheaf $\cala$.
On a complex manifold $M$ of complex dimension $n$,
the analogue of the Poincar\'e lemma is
the $\bar{\partial}$-Poincar\'e lemma \cite[p.~25]{griffiths--harris},
from which it follows that
for each fixed integer $p \ge 0$, the sheaves $\cala^{p,q}$ of smooth
$(p,q)$-forms on $M$ give rise to a resolution of the sheaf
$\Omega^p$ of
holomorphic $p$-forms on $M$:
\begin{equation} \label{6e:holomorphic}
0 \to {\Omega^p} \to \cala^{p,0} \overset{\bar{\partial}}{\to}
\cala^{p,1} \overset{\bar{\partial}}{\to} \cdots
\overset{\bar{\partial}}{\to} \cala^{p,n} \to 0.
\end{equation}
The cohomology of the \term{Dolbeault complex}\index{Dolbeault complex}
\[
0 \to \cala^{p,0}(M) \overset{\bar{\partial}}{\to}
\cala^{p,1}(M) \overset{\bar{\partial}}{\to} \cdots
\overset{\bar{\partial}}{\to} \cala^{p,n}(M) \to 0
\]
of smooth $(p,q)$-forms on $M$
is by definition the \term{Dolbeault cohomology}\index{Dolbeault cohomology} $H^{p,q}(M)$
of the complex manifold $M$.
(For $(p,q)$-forms on a complex manifold, see \cite{griffiths--harris}
or Chapter~\ref{ch:cattani_kahler}.)

\section{Sheaf Cohomology}\label{ch2:sec2}

The \term{de Rham cohomology}\index{de Rham cohomology} $H_{\text{dR}}^*(M)$ of a smooth $n$-manifold
$M$ is defined to be the cohomology of the \term{de Rham complex}\index{de Rham complex}
\[
0 \to \cala^0(M) \to \cala^1(M) \to \cala^2(M) \to \cdots \to
\cala^n(M) \to 0
\]
of $\cinf$-forms on $M$.
De Rham's theorem\index{de Rham theorem} for a smooth manifold $M$ of
dimension $n$ gives an
isomorphism between the real singular cohomology $H^k(M,\R)$
and the de Rham cohomology of $M$ (see \cite[Th.~14.28, p.~175; Th.~15.8, p.~191]{bott--tu}).
One obtains the de Rham complex $\cala\bu(M)$ by
applying the global sections functor $\Gamma(M, \ )$ to the
resolution
\[
0 \to \underline{\R} \to \cala^0 \to \cala^1 \to \cala^2 \to \cdots \to
\cala^n \to 0
\]
of $\underline{\R}$,
but omitting the initial term $\Gamma(M, \underline{\R})$.
This suggests that the cohomology of a sheaf $\calf$ might
be defined as the cohomology of the complex of
global sections of a certain resolution of $\calf$.
Now every sheaf has a canonical resolution:
its \emph{Godement resolution}.
Using the Godement resolution, we will obtain a
well-defined cohomology theory of sheaves.

\subsection{Godement's Canonical Resolution}\label{ss:godement}

Let $\calf$ be a sheaf of abelian groups on a topological
space $X$.
In Section~\ref{ss:associated}, we defined the \'etal\'e
space $E_{\calf}$ of $\calf$.
By Exercise~\ref{1exer:sheafification}, for any open set $U \subset X$, the group $\calf(U)$ may
be interpreted as
\[
\calf(U) = \calf^+(U) = \{ \text{continuous sections of } \pi\colon E_{\calf} \to X \}.
\]
Let $\calc^0\calf (U)$ be the group of all (not necessarily
continuous) sections of the \'etal\'e space $E_{\calf}$ over
$U$; in other words, $\calc^0\calf(U)$ is the direct product $\prod_{p\in U} \calf_p$.
In the literature, $\calc^0\calf$ is often called the sheaf
of \term{discontinuous sections}\index{discontinuous sections} of the \'etal\'e space $E_{\calf}$
of $\calf$.
Then $\calf^+ \simeq \calf$ is a subsheaf of $\calc^0 \calf$ and there is an
exact sequence
\begin{equation}\label{e:canonical0}
0 \to \calf \to \calc^0\calf \to \calq^1 \to 0,
\end{equation}
where $\calq^1$ is the quotient sheaf $\calc^0\calf/\calf$.
Repeating this construction yields exact sequences
\begin{align}
0 \to \calq^1 \to &\calc^0\calq^1 \to \calq^2 \to 0,\label{e:canonical2}\\
0 \to \calq^2 \to &\calc^0\calq^2 \to \calq^3 \to
0,\label{e:canonical3}\\
& \cdots . \notag
\end{align}

The short exact sequences \eqref{e:canonical0} and
\eqref{e:canonical2}
can be spliced together to form a longer exact sequence
\[
\xymatrix{
0 \ar[r] & \calf \ar[r] & \calc^0\calf \ar[rr] \ar@{->>}[dr] && \calc^1\calf \ar[r]
&\calq^2 \ar[r] & 0 \\
& & & \calq^1 \ar@{^{(}->}[ur] & & &
}
\]
with $\calc^1 \calf := \calc^0 \calq^1$.
Splicing together all the short exact sequences
\eqref{e:canonical0}, \eqref{e:canonical2}, \eqref{e:canonical3},
$\ldots,$ and
defining $\calc^k \calf := \calc^0 \calq^k$
results in the long exact sequence
\[
0 \to \calf \to \calc^0\calf \to \calc^1\calf \to \calc^2\calf \to
\cdots ,
\]
called the \term{Godement canonical resolution}\index{Godement canonical resolutions} of $\calf$.
The sheaves $\calc^k\calf$ are called the \term{Godement sheaves}\index{Godement sheaves}
of $\calf$.
(The letter ``$\calc$'' stands for ``canonical.'')

Next we show that the Godement resolution $\calf \to
\calc\bu\calf$
is functorial:\index{functoriality of Godement resolution}\index{Godement resolution!is functorial}
 a sheaf map $\varphi\colon \calf \to \calg$
induces a morphism $\varphi_*\colon \calc\bu\calf \to \calc\bu\calg$
of their Godement resolutions satisfying the two functorial properties:
preservation of the identity and of composition.

A sheaf morphism (sheaf map) $\varphi\colon \cale \to \calf$ induces
a sheaf morphism
\[
\begin{xy}
{\ar (0,12)*!<16pt,0pt>++{\calc^0\varphi\colon\ \ \calc^0\cale}="a";%
(25,12)*++{\calc^0\calf}="b"};%
(0,0)*++{\prod \cale_p}="c";%
(25,0)*++{\prod \calf_p}="d";%
"a"; "c" **\dir2{-};
"b"; "d" **\dir2{-}
\end{xy}
\]
and therefore a morphism of quotient sheaves
\[
\begin{xy}
{\ar (0,12)*++{\calc^0\cale/\cale}="a";%
(25,12)*!<-3pt,0pt>++{\calc^0\calf/\calf\ ,}="b"};%
(0,0)*++{\calq_{\cale}^1}="c";%
(25,0)*++{\calq_{\calf}^1}="d";%
"a"; "c" **\dir2{-};
"b"; "d" **\dir2{-}
\end{xy}
\]
which in turn induces a sheaf morphism
\[
\begin{xy}
{\ar (0,12)*!<14pt,0pt>++{\calc^1\varphi\colon\ \ \calc^0\calq_{\cale}^1}="a";%
(25,12)*!<-6pt,0pt>++{\calc^0\calq_{\calf}^1\ .}="b"};%
(0,0)*++{\calc^1\cale}="c";%
(25,0)*++{\calc^1\calf}="d";%
"a"; "c" **\dir2{-};
"b"; "d" **\dir2{-}
\end{xy}
\]
By induction, we obtain $\calc^k\varphi\colon \calc^k\cale \to \calc^k\calf$ for all $k$.
It can be checked that each $\calc^k(\ )$ is a functor from sheaves to sheaves,
called the \term{$k$th Godement functor}.\index{Godement functors}

Moreover, the induced morphisms $\calc^k\varphi$ fit into a commutative
diagram
\[
\xymatrix{
0 \ar[r] & \cale \ar[r] \ar[d]& \calc^0\cale \ar[r] \ar[d]& \calc^1\cale \ar[r] \ar[d]
& \calc^2\cale \ar[r] \ar[d]& \cdots \\
0 \ar[r] & \calf \ar[r] & \calc^0\calf \ar[r] & \calc^1\calf \ar[r]
& \calc^2\calf \ar[r] & \cdots ,
}
\]
so that collectively $(\calc^k \varphi)_{k=0}^{\infty}$ is a
morphism of Godement resolutions.

\begin{prop} \label{p:godement}
If
\[
0 \to \cale \to \calf \to \calg \to 0
\]
is a short exact sequence of sheaves on a topological
space $X$ and $\calc^k(\ )$ is the $k$th Godement sheaf functor,
then the sequence of sheaves
\[
0 \to \calc^k\cale \to \calc^k\calf \to
 \calc^k\calg \to 0
\]
is exact.
\end{prop}

We say that the Godement functors $\calc^k(\ )$
are \term{exact functors} from sheaves to sheaves.\index{Godement functors!are exact functors}

\pf
For any point $p \in X$, the stalk $\cale_p$ is a subgroup
of the stalk $\calf_p$ with quotient group $\calg_p =
\calf_p/\cale_p$.
Interpreting $\calc^0\cale (U)$ as the direct product
$\prod_{p \in U} \cale_p$ of stalks over $U$,
it is easy to verify that for any open set $U \subset X$,
\begin{equation} \label{e:exact}
0 \to \calc^0\cale(U) \to \calc^0\calf(U) \to
 \calc^0\calg(U) \to 0
\end{equation}
is exact.
In general, the direct limit of exact sequences is exact
\cite[Ch.~2, Exer.~19, p.~33]{atiyah--macdonald}.
Taking the direct limit of \eqref{e:exact} over all
neighborhoods of a point $p \in X$, we obtain the
exact sequence of stalks
\[
0 \to (\calc^0\cale)_p \to (\calc^0\calf)_p \to
 (\calc^0\calg)_p \to 0
\]
for all $p \in X$.
Thus, the sequence of sheaves
\[
0 \to \calc^0\cale \to \calc^0\calf \to
 \calc^0\calg \to 0
\]
is exact.

Let $\calq_{\cale}$ be the quotient sheaf $\calc^0\cale/\cale$,
and similarly for $\calq_{\calf}$ and $\calq_{\calg}$.
Then there is a commutative diagram
\begin{equation} \label{e:nine}
\bfig
\xymatrix{
& 0 \ar[d] & 0 \ar[d] & 0 \ar[d] & \\
0 \ar[r] & \cale \ar[d] \ar[r] & \calc^0\cale \ar[d] \ar[r] &
\calq_{\cale} \ar[d] \ar[r] & 0 \\
0 \ar[r] & \calf \ar[d] \ar[r] & \calc^0\calf \ar[d] \ar[r] &
\calq_{\calf} \ar[d] \ar[r] & 0 \\
0 \ar[r] & \calg \ar[d] \ar[r] & \calc^0\calg \ar[d] \ar[r] &
\calq_{\calg} \ar[d] \ar[r] & 0, \\
& 0 & 0 & 0 &
}
\efig
\end{equation}
in which the three rows and the first two columns are
exact.
It follows by the nine lemma that the last
column is also exact.\footnote{To prove the nine lemma, view each column as
a differential complex.
Then the diagram~\eqref{e:nine} is a short exact
sequence of complexes.
Since the cohomology groups of the first two columns are
zero, the long exact cohomology sequence of the
short exact sequence implies that the cohomology
of the third column is also zero \cite[Th.~25.6, p.~285]{2tu_m}.}
Taking $\calc^0(\ )$ of the last column,
we obtain an exact sequence
\[
\xymatrix@R=10pt{
0 \ar[r] & \calc^0\calq_{\cale} \ar@{=}[d] \ar[r]
& \calc^0\calq_{\calf} \ar@{=}[d] \ar[r]
& \calc^0\calq_{\calg} \ar@{=}[d] \ar[r] & 0. \\
& \calc^1\cale & \calc^1\calf & \calc^1\calg&
}
\]

The Godement resolution is created by alternately taking $\calc^0$ and
taking quotients.
We have shown that each of these two operations preserves exactness.
Hence, the proposition follows by induction.
 \epf

\subsection{Cohomology with Coefficients in a Sheaf}

Let $\calf$ be a sheaf of abelian groups on a topological space $X$.
What is so special about the Godement resolution of $\calf$
is that it is completely canonical.
For any open set $U$ in $X$,
applying the sections functor $\Gamma(U,\ )$
to the Godement resolution of $\calf$ gives a complex
\begin{equation} \label{2e:complex}
0 \to \calf(U) \to \calc^0\calf(U) \to
\calc^1\calf(U) \to
\calc^2\calf(U) \to \cdots .
\end{equation}
In general, the $k$th \term{cohomology}\index{cohomology!of a complex} of a complex
\[
0 \to K^0 \overset{d}{\to} K^1 \overset{d}{\to} K^2 \to \cdots
\]
 will be denoted by
\[
h^k (K\bu) := \frac{\ker (d\colon K^k \to K^{k+1})}{\im (d\colon K^{k-1}
\to K^k)}.
\]
We sometimes write a complex $(K\bu, d)$ not as a sequence,
but as a direct sum $K\bu = \bigoplus_{k=0}^{\infty} K^k$,
with the understanding that $d\colon K^k \to K^{k+1}$
increases the degree by $1$ and $d \comp d = 0$.
The \term{cohomology of $U$ with coefficients in the sheaf
$\calf$},\index{cohomology!with coefficients in a sheaf}
 or the \term{sheaf cohomology of $\calf$ on $U$},\index{sheaf cohomology} is defined to be
the cohomology of the complex $\calc\bu \calf(U) =
\bigoplus_{k \ge 0} \calc^k\calf(U)$ of sections of the Godement
resolution of $\calf$
 (with the initial term $\calf(U)$ dropped from the
 complex \eqref{2e:complex}):
\[
H^k(U, \calf) := h^k \big(\calc\bu \calf(U)\big).
\]

\begin{prop} \label{2p:sections}
Let $\calf$ be a sheaf on a topological space $X$.
For any open set $U \subset X$, we have $H^0(U,\calf)
= \Gamma(U, \calf)$.
\end{prop}%
\index{cohomology!in degree $0$}

\pf
If
\[
0 \to \calf \to \calc^0\calf \to \calc^1\calf \to \calc^2\calf \to
\cdots
\]
is the
Godement resolution of $\calf$, then by definition
\[
H^0(U, \calf) = \ker \big(d\colon \calc^0\calf(U) \to
  \calc^1\calf(U)\big).
\]
In the notation of the preceding subsection,
$d\colon \calc^0\calf(U) \to
\calc^1\calf(U)$ is induced from the composition
of sheaf maps
\[
\calc^0 \calf  \twoheadrightarrow \calq^1 \hookrightarrow \calc^1\calf.
\]
Thus, $d\colon \calc^0\calf(U) \to
\calc^1\calf(U)$ is the composition of
\[
\calc^0\calf(U) \to \calq^1(U) \hookrightarrow
\calc^1\calf(U).
\]
Note that the second map $\calq^1(U) \hookrightarrow
\calc^1\calf(U)$ is injective, because $\Gamma(U, \ )$
is a left-exact functor.
Hence,
\begin{align*}
H^0(U, \calf) &= \ker \big(\calc^0\calf(U) \to \calc^1\calf(U)\big)\\
&= \ker \big(\calc^0\calf(U) \to \calq^1(U)\big).
\end{align*}
But from the exactness of
\[
0 \to \calf(U) \to \calc^0\calf(U) \to
\calq^1(U),
\]
we see that
\[
\Gamma(U,\calf) = \calf(U) = \ker \big(\calc^0\calf(U) \to \calq^1(U) \big) =
H^0(U, \calf). 
\]
\epf

\subsection{Flasque Sheaves}

Flasque sheaves are a special kind of sheaf with
vanishing higher cohomology.  All Godement sheaves
turn out to be flasque sheaves.

\begin{defi}
A sheaf $\calf$ of abelian groups on a topological space
$X$ is \term{flasque}%
\index{flasque sheaves} (French for ``flabby'') if for every
open set $U \subset X$, the restriction map
$\calf(X) \to \calf(U)$ is surjective.
\end{defi}

For any sheaf $\calf$, the Godement sheaf $\calc^0\calf$
is clearly flasque because $\calc^0\calf(U)$ consists
of all discontinuous sections of the \'etal\'e space $E_{\calf}$
over $U$.
In the notation of the preceding subsection, $\calc^k\calf =
\calc^0\calq^k$, so all Godement sheaves $\calc^k\calf$ are flasque.%
\index{Godement sheaves!are flasque}

\begin{prop} \label{p:flasque}
\begin{enumerate}[(i)]
\item\label{ch2:prp224.i} In a short exact sequence of sheaves
\begin{equation} \label{e:short}
0 \to \cale \overset{i}{\to} \calf \overset{j}{\to} \calg \to 0
\end{equation}
over a topological space $X$, if $\cale$ is flasque, then
for any open set $U \subset X$,
the sequence of abelian groups
\[
0 \to \cale(U) \to \calf(U) \to \calg(U) \to 0
\]
is exact.
\item\label{ch2:prp224.ii} If $\cale$ and $\calf$ are flasque in \eqref{e:short},
then $\calg$ is flasque.
\item\label{ch2:prp224.iii} If
\begin{equation} \label{e:sheaves}
0 \to \cale \to \call^0 \to \call^1 \to \call^2 \to \cdots
\end{equation}
is an exact sequence of flasque sheaves on $X$, then for any
open set $U \subset X$ the sequence of abelian groups of sections
\begin{equation} \label{e:groups}
0 \to \cale(U) \to \call^0(U) \to \call^1(U) \to \call^2(U) \to \cdots
\end{equation}
is exact.
\end{enumerate}
\end{prop}

\pf
\begin{enumerate}[(i)]
\item To simplify the notation, we will use $i$ to denote $i_U\colon
\cale(U) \to \calf(U)$ for all $U$; similarly, $j = j_U$.
As noted in Section~\ref{s:exact}, the exactness of
\begin{equation} \label{e:sections}
0 \to \cale(U) \overset{i}{\to} \calf(U) \overset{j}{\to} \calg(U)
\end{equation}
is true in general, whether $\cale$ is flasque or not.
To prove the surjectivity of $j$ for a flasque $\cale$, let $g \in \calg(U)$.
Since $\calf \to \calg$ is surjective as a sheaf map,
all stalk maps $\calf_p \to \calg_p$ are surjective.
Hence, every point $p \in U$ has a neighborhood $U_{\ga} \subset U$
on which there exists a section $f_{\ga} \in \calf(U_{\ga})$ such
that $j(f_{\ga}) = g|_{U_{\ga}}$.

Let $V$ be the largest union $\bigcup_{\ga} U_{\ga}$ on which
there is a section $f_V \in \calf(V)$ such that $j(f_V) = g|_V$.
We claim that $V = U$.
If not, then there are a set $U_{\ga}$ not contained in $V$ and $f_{\ga} \in
\calf(U_{\ga})$ such that $j(f_{\ga}) = g|_{U_{\ga}}$.
On $V \cap U_{\ga}$, writing $j$ for $j_{V \cap U_{\ga}}$, we have
\[
j(f_V - f_{\ga}) =0.
\]
By the exactness of the sequence \eqref{e:sections} at $\calf(V \cap U_{\ga})$,
\[
f_V - f_{\ga} = i(e_{V,\ga}) \ \text{for some } e_{V,\ga} \in \cale(V
\cap U_{\ga}).
\]
Since $\cale$ is flasque, one can find a section $e_U
\in \cale(U)$ such that $e_U|_{V \cap U_{\ga}} =
e_{V,\ga}$.

On $V \cap U_{\ga}$,
\[
f_V = i(e_{V,\ga}) + f_{\ga}.
\]
If we modify $f_{\ga}$ to
\[
\bar{f}_{\ga} = i(e_U) + f_{\ga} \quad \text{on }
U_{\ga},
\]
then $f_V = \bar{f}_{\ga}$ on $V \cap U_{\ga}$,
and $j(\bar{f}_{\ga})=g|_{U_{\ga}}$.
By the gluing axiom for the sheaf $\calf$,
the elements $f_V$ and $\bar{f}_{\ga}$ piece together to give an element
$f \in \calf(V \cup U_{\ga})$ such that
$j(f) = g|_{V \cup U_{\ga}}$.
This contradicts the maximality of $V$.
Hence, $V = U$ and $j\colon \calf(U) \to \calg(U)$ is onto.

\item Since $\cale$ is flasque, for any open set $U \subset X$
the rows of the commutative diagram
\[
\xymatrix{
0 \ar[r] & \cale(X) \ar[d]_-{\ga} \ar[r] & \calf(X) \ar[d]^-{\beta}
\ar[r]^-{j_X} & \calg(X) \ar[d]^-{\gamma} \ar[r] & 0 \\
0 \ar[r] & \cale(U)  \ar[r] & \calf(U)
\ar[r]^-{j_U} & \calg(U)  \ar[r] & 0
}
\]
are exact by \ref{ch2:prp224.i}, where $\ga$, $\gb$, and $\gamma$ are the restriction maps.
Since $\calf$ is flasque, the map $\beta\colon \calf(X) \to
\calf(U)$ is surjective.
Hence,
\[
j_U \comp \beta = \gamma \comp j_X \colon \calf(X) \to \calg(X) \to \calg(U)
\]
is surjective.
Therefore, $\gamma\colon \calg(X) \to \calg(U)$ is surjective.
This proves that $\calg$ is flasque.

\item  The long exact sequence \eqref{e:sheaves} is equivalent to a
collection
of short exact sequences
\begin{align}
0 \to \cale \to &\call^0 \to \calq^0 \to 0, \label{e:flasque1}\\
0 \to \calq^0 \to &\call^1 \to \calq^1 \to 0, \label{e:flasque2}\\
& \cdots  . \notag
\end{align}
In \eqref{e:flasque1}, the first two sheaves are flasque,
so $\calq^0$ is flasque by \ref{ch2:prp224.ii}.
Similarly, in \eqref{e:flasque2}, the first two sheaves are flasque,
so $\calq^1$ is flasque.
By induction, all the sheaves $\calq^k$ are flasque.

By \ref{ch2:prp224.i}, the functor $\Gamma(U,\ )$ transforms the short
exact sequences of sheaves into short exact sequences
of abelian groups
\begin{align*}
0 \to \cale(U) \to &\call^0(U) \to \calq^0(U) \to 0,\\
0 \to \calq^0(U) \to &\call^1(U) \to \calq^1(U) \to 0, \\
& \cdots  . \notag
\end{align*}
These short exact sequences splice together into the
long exact sequence \eqref{e:groups}.
\end{enumerate}
 \epf

\begin{cor}\label{c:flasque_acyclic}
Let $\cale$ be a flasque sheaf on a topological space $X$.
For every open set $U \subset X$ and every $k > 0$,
the cohomology $H^k(U, \cale)= 0$.\index{cohomology!of a flasque sheaf}\index{flasque
sheaves!cohomology of}
\end{cor}

\pf
Let
\[
0 \to \cale \to \calc^0\cale \to \calc^1\cale \to \calc^2\cale \to
\cdots
\]
be the Godement resolution of $\cale$.
It is an exact sequence of flasque sheaves.
By Proposition~\ref{p:flasque}\ref{ch2:prp224.iii}, the sequence of groups of
sections
\[
0 \to \cale(U)\to \calc^0\cale(U) \to \calc^1\cale(U) \to \calc^2\cale(U) \to
\cdots
\]
is exact.  It follows from the definition of sheaf cohomology that
\[
H^k(U,\cale) =\begin{cases}
\ \cale(U) &\text{for } k = 0, \\
\ 0 & \text{for } k > 0. 
\end{cases}
\]
\epf

A sheaf $\calf$ on a topological space $X$ is said to be
\term{acyclic}\index{acyclic sheaf!on an open set} on $U \subset X$ if $H^k(U, \calf) =0$ for
all $k > 0$.
Thus,
a flasque sheaf on $X$ is acyclic on every open set of $X$.\index{flasque sheaves!are acyclic}

\begin{exa} \label{exam:irreducible}
Let $X$ be an irreducible complex algebraic variety with the Zariski topology.
Recall that the constant sheaf $\underline{\C}$ over $X$ is the
sheaf of locally constant functions on $X$ with values in $\C$.
Because any two open sets in the Zariski topology of $X$ have a nonempty
intersection, the only continuous sections of the constant
sheaf $\underline{\C}$ over any open set $U$ are the constant functions.
Hence, $\underline{\C}$ is flasque.
By Corollary~\ref{c:flasque_acyclic}, $H^k(X, \underline{\C}) = 0$ for all $k > 0$.%
\index{cohomology!of a constant sheaf with Zariski\\ topology}%
\index{constant sheaf!cohomology of}
\end{exa}

\begin{cor} \label{c:functor}
Let $U$ be an open subset of a topological space $X$.
The $k$th Godement sections
functor $\Gamma(U, \calc^k(\ ))$,\index{Godement sections functor} which assigns to a sheaf $\calf$
on $X$ the group $\Gamma(U, \calc^k\calf)$ of sections
of $\calc^k\calf$ over $U$, is an exact functor from sheaves on $X$
to abelian groups.\index{Godement sections functor!is an exact functor}
\end{cor}

\pf
Let
\[
0 \to \cale \to \calf \to \calg \to 0
\]
be an exact sequence of sheaves.
By Proposition~\ref{p:godement}, for any $k \ge 0$,
\[
0 \to \calc^k\cale \to \calc^k\calf \to \calc^k\calg \to 0
\]
is an exact sequence of sheaves.
Since $\calc^k\cale$ is flasque, by
Proposition~\ref{p:flasque}\ref{ch2:prp224.i},
\[
0 \to \Gamma(U, \calc^k\cale) \to \Gamma(U, \calc^k\calf)
\to \Gamma(U, \calc^k\calg) \to 0
\]
is an exact sequence of abelian groups.
Hence, $\Gamma\big(U, \calc^k(\ )\big)$ is an exact functor
from sheaves to groups.
 \epf

Although we do not need it, the following theorem is a fundamental
property of sheaf cohomology.

\begin{theorem} \label{t:long}
A short exact sequence
\[
0 \to \cale \to \calf \to \calg \to 0
\]
of sheaves of abelian groups on a topological space $X$
induces a long exact sequence in sheaf cohomology,
\[
\cdots \to H^k(X, \cale) \to H^k(X, \calf) \to H^k(X, \calg) \to
H^{k+1}(X, \cale) \to \cdots .
\]
\end{theorem}

\pf
Because the Godement sections functor $\Gamma\big(X, \calc^k(\ )\big)$ is
exact,
from the given short exact sequence of sheaves one obtains a short
exact sequence of complexes of global sections of Godement sheaves
\[
0 \to \calc\bu\cale(X) \to \calc\bu\calf(X) \to \calc\bu\calg(X) \to 0.
\]
The long exact sequence in cohomology
\cite[Sec.~25]{2tu_m} associated to this short exact
sequence of complexes is the desired long exact sequence in sheaf
cohomology.
 \epf

\subsection{Cohomology Sheaves and Exact Functors}\label{ss:cohomsheaves}
As before, a sheaf will mean a sheaf of abelian groups
on a topological space $X$.
A \term{complex of sheaves}\index{complex of sheaves} $\call^{\bullet}$ on $X$ is a sequence
of sheaves
\[
0 \to \call^0 \overset{d}{\to} \call^1 \overset{d}{\to} \call^2 \overset{d}{\to} \cdots
\]
on $X$ such that $d \comp d = 0$.
Denote the kernel and image sheaves of $\call^{\bullet}$ by
\begin{align*}
\calz^k &:= \calz^k(\call^{\bullet}) := \ker \big(d\colon \call^k \to \call^{k+1}\big),\\
\calb^k &:= \calb^k(\call^{\bullet}) := \im \big(d\colon \call^{k-1} \to \call^k\big).
\end{align*}
Then the \term{cohomology sheaf}\index{cohomology sheaves!of a complex of sheaves}%
\index{complex of sheaves!cohomology sheaves of}
 $\calh^k := \calh^k (\call^{\bullet})$ of the complex $\call^{\bullet}$ is the quotient sheaf
\[
\calh^k := \calz^k /\calb^k.
\]
For example, by the Poincar\'e lemma, the complex
\[
0 \to \cala^0 \to \cala^1 \to \cala^2 \to \cdots
\]
of sheaves of $\cinf$-forms on a manifold $M$ has cohomology sheaves
\[
\calh^k = \calh^k(\cala^{\bullet}) = \begin{cases}
\ \underline{\R} &\text{for } k=0,\\
\ 0 &\text{for } k > 0.
\end{cases}
\]

\begin{prop} \label{p:stalkcohomology}
Let $\call\bu$ be a complex of sheaves on a topological space $X$.
The stalk of its cohomology sheaf $\calh^k$ at a point $p$ is
the $k$th cohomology of the complex $\call_p\bu$ of stalks.%
\index{stalk!of a cohomology sheaf}\index{cohomology sheaves!stalk}
\end{prop}

\pf
Since
\[
\calz_p^k = \ker \big(d_p\colon \call_p^k \to \call_p^{k+1}\big)
\ \text{ and }\ \calb_p^k =
\im \big(d_p\colon \call_p^{k-1} \to \call_p^k \big)
\]
(see \cite[Ch.~II, Exer.~1.2(a), p.~66]{hartshorne}),
one can also compute the stalk of the cohomology sheaf $\calh^k$ by
computing
\[
\calh_p^k = (\calz^k/\calb^k)_p = \calz_p^k /\calb_p^k =
h^k(\call_p\bu),
\]
the cohomology of the sequence of stalk maps of $\call\bu$ at $p$.
 \epf

Recall that a \term{morphism}\index{morphism!of complexes of sheaves} $\varphi\colon \calf\bu \to \calg\bu$
of complexes of sheaves is a collection of sheaf maps
$\varphi^k\colon \calf^k \to \calg^k$ such that
$\varphi^{k+1} \comp d = d \comp \varphi^k$ for all $k$.
A morphism $\varphi\colon \calf\bu \to \calg\bu$ of complexes
of sheaves induces morphisms
$\varphi^k\colon \calh^k(\calf\bu) \to \calh^k(\calg\bu)$
of cohomology sheaves.
The morphism $\varphi\colon \calf\bu \to \calg\bu$ of complexes
of sheaves is called a \term{quasi-isomorphism}\index{quasi-isomorphism}%
\index{morphism!of complexes of sheaves!quasi-isomorphism} \label{p:quasi}
if the induced morphisms $\varphi^k\colon \calh^k(\calf\bu)\to
\calh^k(\calg\bu)$ of cohomology sheaves are isomorphisms
for all $k$.

\begin{prop} \label{p:exact}
Let $\call\bu = \bigoplus_{k \ge 0} \call^k$ be a complex of sheaves on a topological space $X$.
If $T$ is an exact functor from sheaves on $X$
to abelian groups, then it commutes with cohomology:\index{exact functor!commutes with cohomology}
\[
T \big(\calh^k(\call\bu)\big) = h^k \big(T(\call\bu)\big).
\]
\end{prop}

\pf
We first prove that $T$ commutes with cocycles and coboundaries.
Applying the exact functor $T$ to the exact sequence
\[
0 \to \calz^k \to \call^k \overset{d}{\to} \call^{k+1}
\]
results in the exact sequence
\[
0 \to T(\calz^k) \to T(\call^k) \overset{d}{\to} T(\call^{k+1}),
\]
which proves that
\[
Z^k\big(T(\call\bu)\big) := \ker\big( T(\call^k) \overset{d}{\to}
T(\call^{k+1}) \big) = T(\calz^k).
\]
(By abuse of notation, we write the differential of $T(\call\bu)$ also
as $d$, instead of $T(d)$.)

The differential $d\colon \call^{k-1} \to \call^k$ factors into
a surjection $\call^{k-1} \twoheadrightarrow \calb^k$ followed by an injection
$\calb^k \hookrightarrow \call^k$:
\[
\xymatrix{
\call^{k-1} \ar[rr]^-d \ar@{->>}[dr]&&\call^k. \\
& \calb^k \ar@{^{(}->}[ur]
}
\]
Since an exact functor preserves surjectivity and injectivity,
applying $T$ to the diagram above yields a commutative
diagram
\[
\xymatrix{
T(\call^{k-1}) \ar[rr]^-d \ar@{->>}[dr]&&T(\call^k), \\
& T(\calb^k) \ar@{>->}[ur]
}
\]
which proves that
\[
B^k\big(T(\call\bu)\big) := \im\big( T(\call^{k-1}) \overset{d}{\to}
T(\call^k) \big) = T(\calb^k).
\]

Applying the exact functor $T$ to the exact sequence
of sheaves
\[
0  \to \calb^k \to \calz^k \to \calh^k \to 0
\]
gives the exact sequence of abelian groups
\[
0  \to T(\calb^k) \to T(\calz^k) \to T(\calh^k) \to 0.
\]
Hence,
\[
T\big(\calh^k(\call\bu)\big) = T(\calh^k)
= \frac{T(\calz^k)}{T(\calb^k)}
= \frac{Z^k\big(T(\call\bu)\big)}{B^k\big(T(\call\bu)\big)}
= h^k\big(T(\call\bu)\big). 
\]
\epf

\subsection{Fine Sheaves} \label{ss:fine}

We have seen that flasque sheaves on a topological space $X$
are acyclic on any open subset of $X$.
Fine sheaves constitute another important class of such sheaves.

A sheaf map $f\colon \calf \to \calg$ over a topological space
$X$ induces at each point $x \in X$ a group homomorphism
$f_x\colon \calf_x \to \calg_x$ of stalks.  The \term{support}\index{support!of a sheaf morphism}%
\index{sheaf morphism!support of} of
the sheaf morphism $f$ is defined to be
\[
\supp f = \{ x \in X \mid f_x \ne 0 \}.
\]

If two sheaf maps over a topological space $X$ agree
at a point, then they agree in a neighborhood
of that point, so the set where two sheaf maps agree is open
in $X$.
Since the complement $X - \supp f$ is the subset of $X$
where the sheaf map $f$ agrees with the zero sheaf map,
it is open and therefore $\supp f$ is closed.

\begin{defi}
Let $\calf$ be a sheaf of abelian groups on a topological space $X$
and $\{ U_{\ga} \}$ a locally finite open cover of $X$.
A \term{partition of unity}\index{partition of unity!of a sheaf} of $\calf$ subordinate to $\{ U_{\ga}\}$
is a collection  $\{ \eta_{\ga}\colon \calf \to \calf \}$ of sheaf maps
such that
\begin{enumerate}[(i)]
\item   $\supp \eta_{\ga} \subset U_{\ga}$;
\item\label{ch2:def2211.ii}   for each point $x \in X$, the sum $\sum \eta_{\ga,x} = \id_{\calf_x}$, the identity map on the stalk $\calf_x$.
\end{enumerate}
\end{defi}

Note that although $\ga$ may range over an infinite index set,
the sum in \ref{ch2:def2211.ii} is a finite sum, because $x$ has a neighborhood that
meets only finitely many of the $U_{\ga}$'s and $\supp \eta_{\ga} \subset
U_{\ga}$.

\begin{defi}
A sheaf $\calf$ on a topological space $X$ is said to be \term{fine}\index{fine sheaf}
if for every locally finite open cover $\{ U_{\ga}\}$ of $X$,
the sheaf $\calf$ admits a partition of unity subordinate to $\{
U_{\ga}\}$.
\end{defi}

\begin{prop} \label{2p:fine}
The sheaf $\cala^k$ of smooth $k$-forms on a manifold
$M$ is a fine sheaf on $M$.
\end{prop}

\pf
Let $\{ U_{\ga}\}$ be a locally finite open cover of $M$.
Then there is a $\cinf$ partition of unity $\{ \rho_{\ga} \}$
on $M$ subordinate to $\{ U_{\ga} \}$ \cite[App.~C, p.~346]{2tu_m}.
(This partition of unity  $\{ \rho_{\ga} \}$ is a collection of smooth $\R$-valued functions,
not sheaf maps.)
For any open set $U \subset M$, define
$\eta_{\ga,U}\colon \cala^k(U) \to \cala^k(U)$ by
\[
\eta_{\ga,U}(\omega) = \rho_{\ga} \omega.
\]

If $x \notin U_{\ga}$, then $x$ has a neighborhood $U$
disjoint from $\supp \rho_{\ga}$.
Hence, $\rho_{\ga}$ vanishes identically on $U$
and $\eta_{\ga,U} = 0$,
so that the stalk map $\eta_{\ga,x}\colon \cala_x^k \to \cala_x^k$ is the zero map.
This proves that $\supp \eta_{\ga} \subset U_{\ga}$.

For any $x\in M$, the stalk map $\eta_{\ga,x}$ is multiplication
by the germ of $\rho_{\ga}$,
so $\sum_{\ga} \eta_{\ga,x}$ is the identity map on the stalk $\cala_x^k$.
Hence,  $\{ \eta_{\ga} \}$ is a partition of unity of the sheaf
$\cala^k$
subordinate to $\{ U_{\ga} \}$.
 \epf

Let $\calr$ be a sheaf of commutative rings on a topological space
$X$.
A sheaf $\calf$ of abelian groups on $X$ is called a
\term{sheaf of $\calr$-modules}\index{sheaf of modules!over a sheaf of commutative rings} (or simply an
\term{$\calr$-module}\index{r-module@$\calr$-module}) if for every open set $U \subset X$,
the abelian group $\calf(U)$ has an $\calr(U)$-module structure and
moreover, for all $V \subset U$, the restriction map $\calf(U) \to
\calf(V)$
is compatible with the module structure in the sense that the diagram
\[
\xymatrix{
\calr(U) \times \calf(U) \ar[d] \ar[r] & \calf(U) \ar[d] \\
\calr(V) \times \calf(V) \ar[r] & \calf(V)
}
\]
commutes.

A \term{morphism}\index{morphism!of sheaves of modules} $\varphi\colon \calf \to \calg$ of sheaves of $\calr$-modules
over $X$ is a sheaf morphism such that for each open set $U \subset X$,
the group homomorphism $\varphi_U\colon \calf(U) \to \calg(U)$ is
an $\calr(U)$-module homomorphism.

If $\cala^0$ is the sheaf of $\cinf$ functions on a manifold $M$,
then the sheaf $\cala^k$ of smooth $k$-forms on $M$ is a
sheaf of $\cala^0$-modules.
By a proof analogous to that of Proposition~\ref{2p:fine},
any sheaf of $\cala^0$-modules over a manifold is a fine sheaf.
In particular, the sheaves $\cala^{p,q}$ of smooth $(p,q)$-forms on
a complex manifold are all fine sheaves.

\subsection{Cohomology with Coefficients in a Fine Sheaf}

A topological space $X$ is \term{paracompact}\index{paracompact topological space} if every open cover of
$X$ admits a locally finite open refinement.
In working with fine sheaves, one usually has to assume that the
topological space is paracompact, in order to be assured
of the existence of a locally finite open cover.\index{locally finite open cover}
A common and important class of paracompact spaces
is the class of topological manifolds \cite[Lem.~1.9, p.~9]{warner}.\index{topological manifolds!are paracompact}

A fine sheaf is generally not flasque.
For example, $f(x) = \sec x$ is a $\cinf$ function on
the open interval $U= \ ]-\pi/2,
\pi/2[$
that cannot be extended to a $\cinf$ function on $\R$.
This shows that $\cala^0(\R) \to \cala^0(U)$ is not surjective.
Thus, the sheaf $\cala^0$ of $\cinf$ functions is a fine sheaf that is
not flasque.

While flasque sheaves are useful for defining cohomology,
fine sheaves are more prevalent in differential topology.
Although fine sheaves need not be flasque,
they share many of the properties of flasque sheaves.
For example, on a manifold,
Proposition~\ref{p:flasque} and Corollary~\ref{c:flasque_acyclic}
remain true if the sheaf $\cale$ is fine instead of flasque.

\begin{prop}
\begin{enumerate}[(i)]
\item\label{ch2:prp2214.i}
In a short exact sequence of sheaves
\begin{equation} \label{e:fine}
0 \to \cale \to \calf \to \calg \to 0
\end{equation}
of abelian groups over a paracompact space $X$, if $\cale$ is fine, then
the sequence of abelian groups of global sections
\[
0 \to \cale(X) \overset{i}{\to} \calf(X) \overset{j}{\to} \calg(X) \to 0
\]
is exact.
\end{enumerate}
\noindent
In \ref{ch2:prp2214.ii} and \ref{ch2:prp2214.iii}, assume that
every open subset of $X$ is paracompact (a manifold is an example of
such a space $X$).
\begin{enumerate}[(i)]\setcounter{enumi}{1}
\item\label{ch2:prp2214.ii}  If $\cale$ is fine and $\calf$ is flasque in
\eqref{e:fine}, then $\calg$ is flasque.
\item\label{ch2:prp2214.iii}  If
\[
0 \to \cale \to \call^0 \to \call^1 \to \call^2 \to \cdots
\]
is an exact sequence of sheaves on $X$ in which $\cale$
is fine and all the $\call^k$ are flasque, then for
any open set $U \subset X$, the sequence of abelian groups
\[
0 \to \cale(U) \to \call^0(U) \to \call^1(U) \to \call^2(U) \to \cdots
\]
is exact.
\end{enumerate}
\end{prop}

\pf
To simplify the notation, $i_U\colon \cale(U) \to \calf(U)$ will generally
be denoted by $i$.
Similarly, ``$f_{\ga}$ on $U_{\ga \gb}$'' will mean $f_{\ga}|_{U_{\ga \gb}}$.
As in Proposition~\ref{p:flasque}\ref{ch2:prp224.i}, it suffices to show that if
$\cale$
is a fine sheaf, then $j\colon \calf(X) \to \calg(X)$ is surjective.
Let $g \in \calg(X)$.
Since $\calf_p \to \calg_p$ is surjective for all $p \in X$,
there exist an open cover $\{ U_{\ga} \}$ of $X$ and
elements $f_{\ga} \in \calf(U_{\ga})$ such that
$j(f_{\ga}) = g|_{U_{\ga}}$.
By the paracompactness of $X$, we may assume that the
open cover $\{ U_{\ga} \}$ is locally finite.
On $U_{\ga \gb} := U_{\ga} \cap U_{\gb}$,
\[
j(f_{\ga}|_{U_{\ga \gb}} - f_{\gb}|_{U_{\ga \gb}}) =
j(f_{\ga})|_{U_{\ga \gb}} - j(f_{\gb})|_{U_{\ga \gb}} =
g|_{U_{\ga\gb}} - g|_{U_{\ga\gb}} = 0.
\]

By the exactness of the sequence
\[
0 \to \cale(U_{\ga \gb}) \overset{i}{\to} \calf(U_{\ga \gb})
\overset{j}{\to} \calg(U_{\ga \gb}),
\]
there is an element $e_{\ga \gb} \in \cale(U_{\ga \gb})$ such that
on $U_{\ga \gb}$,
\[
f_{\ga} - f_{\gb} = i(e_{\ga \gb}).
\]
Note that on the triple intersection
$U_{\ga\gb\gc} := U_{\ga} \cap U_{\gb} \cap U_{\gc}$, we have
\[
i(e_{\ga \gb} + e_{\gb \gc}) = f_{\ga} -f_{\gb} + f_{\gb} - f_{\gc} = i(e_{\ga \gc}).
\]

Since $\cale$ is a fine sheaf, it admits a partition of unity
$\{ \eta_{\ga} \}$ subordinate to $\{ U_{\ga} \}$.
We will now view an element of $\cale(U)$ for any open set $U$ as a continuous section
of the \'etal\'e space $E_{\cale}$ over $U$.
Then the section $\eta_{\gamma}(e_{\alpha \gamma}) \in \cale(U_{\ga
  \gc})$
can be extended by zero to a continuous section of $E_{\cale}$ over
$U_{\ga}$:
\[
\overline{\eta_{\gc}e_{\ga \gc}}(p) = \begin{cases}
(\eta_{\gc}e_{\ga \gc})(p) &\text{for } p \in U_{\ga \gc},\\
0 &\text{for } p \in U_{\ga} - U_{\ga \gc}.
\end{cases}
\]
(Proof of the continuity of $\overline{\eta_{\gc}e_{\ga \gc}}$:
On $U_{\ga\gc}$, $\eta_{\gc}e_{\ga \gc}$ is continuous.
If $p \in U_{\ga} - U_{\ga \gc}$, then $p \notin U_{\gc}$, so
$p \notin \supp \eta_{\gc}$.
Since $\supp \eta_{\gc}$ is closed, there is an open set
$V$ containing $p$ such that $V \cap \supp \eta_{\gc} = \varnothing$.
Thus, $\overline{\eta_{\gc}e_{\ga \gc}}= 0$ on $V$, which proves
that $\overline{\eta_{\gc}e_{\ga \gc}}$ is continuous at $p$.)

To simplify the notation, we will omit the overbar and write
$\eta_{\gc}e_{\ga \gc} \in \cale(U_{\ga})$ also for the extension by
zero
of $\eta_{\gc}e_{\ga \gc} \in \cale(U_{\ga\gc})$.
Let $e_{\ga}$ be the locally finite sum
\[
e_{\ga} = \sum_{\gamma} \eta_{\gc} e_{\ga\gc} \in
\cale(U_{\ga}).
\]
On the intersection $U_{\ga\gb}$,
\begin{align*}
i(e_{\ga} - e_{\gb}) &= i\Big(\sum_{\gc} \eta_{\gc}e_{\ga\gc}
- \sum_{\gc} \eta_{\gc}e_{\gb\gc}\Big)
= i \Big(\sum_{\gc} \eta_{\gc} (e_{\ga\gc} - e_{\gb\gc})\Big) \\
&= i\Big(\sum_{\gc} \eta_{\gc}e_{\ga\gb}\Big) = i(e_{\ga\gb}) = f_{\ga}- f_{\gb}.
\end{align*}
Hence, on $U_{\ga\gb}$,
\[
f_{\ga} - i(e_{\ga}) = f_{\gb} - i(e_{\gb}).
\]
By the gluing sheaf axiom for the sheaf $\calf$, there is an element $f\in
\calf(X)$
such that $f|_{U_{\ga}} = f_{\ga} - i(e_{\ga})$.
Then
\[
j(f)|_{U_{\ga}} = j(f_{\ga}) = g|_{U_{\ga}} \ \ \text{for all } \ga.
\]
By the uniqueness sheaf axiom for the sheaf $\calg$, we have $j(f)=g \in
\calg(X)$.
This proves the surjectivity of $j\colon \calf(X) \to \calg(X)$.

\noindent
\ref{ch2:prp2214.ii}, \ref{ch2:prp2214.iii} Assuming that every open subset $U$ of $X$ is paracompact,
we can apply \ref{ch2:prp2214.i} to $U$.  Then the proofs of \ref{ch2:prp2214.ii} and \ref{ch2:prp2214.iii} are the
same
as in Proposition~\ref{p:flasque}\ref{ch2:prp224.ii}, \ref{ch2:prp224.iii}.
 \epf

The analogue of Corollary~\ref{c:flasque_acyclic} for $\cale$ a fine
sheaf then follows as before.
The upshot is the following theorem.

\begin{theorem}
Let $X$ be a topological space in which every open subset is
paracompact.
Then a fine sheaf on $X$ is acyclic
on every open subset $U$.\index{fine sheaf!is acyclic on every open subset}
\end{theorem}

\begin{rem} \label{r:cohomology}
Sheaf cohomology can be characterized uniquely by a set
of axioms \cite[Def.~5.18, pp.~176--177]{warner}.
Both the sheaf cohomology in terms of Godement's resolution
and the \v{C}ech cohomology of a paracompact Hausdorff space
satisfy these axioms \cite[pp.~200--204]{warner},
so at least on a paracompact Hausdorff space,
sheaf cohomology is isomorphic to \v{C}ech cohomology.\index{sheaf cohomology!isomorphism with \v{C}ech cohomology}\index{Cech cohomology@\v{C}ech cohomology!isomorphism with sheaf cohomology}
Since the \v{C}ech cohomology of a triangularizable space
with coefficients in the constant sheaf $\underline{\Z}$ is
isomorphic to its singular cohomology with integer coefficients
\cite[Th.~15.8, p.~191]{bott--tu},
the sheaf cohomology $H^k(M, \underline{\Z})$ of a manifold $M$ is isomorphic
to the singular cohomology $H^k(M, \Z)$.
In fact, the same argument shows that one may replace $\Z$ by $\R$
or by $\C$.
\end{rem}

\section{Coherent Sheaves and Serre's GAGA Principle}

Given two sheaves $\calf$ and $\calg$ on $X$, it is easy to show
that the presheaf $U \mapsto \calf(U) \oplus \calg(U)$ is a sheaf,
called the \term{direct sum}%
\index{direct sum!of sheaves} of $\calf$ and $\calg$ and denoted
by $\calf \oplus \calg$.
We write the direct sum of $p$ copies of $\calf$ as $\calf^{\oplus p}$.
If $U$ is an open subset of $X$, the \term{restriction}\index{restriction!of a sheaf to an open subset} $\calf|_U$
of the sheaf $\calf$ to
$U$ is the sheaf on $U$ defined by $(\calf|_U)(V) = \calf(V)$ for every open subset
$V$ of $U$.
Let $\calr$ be a sheaf of commutative rings on a topological space
$X$.
A sheaf $\calf$ of $\calr$-modules on $M$ is \term{locally free
of rank $p$}\index{locally free sheaf}\index{rank!of a locally free sheaf}
if every point $x \in M$ has a neighborhood $U$ on which
there is a sheaf isomorphism $\calf|_U
\simeq \calr|_U^{\oplus p}$.

Given a complex manifold $M$, let $\calo_M$ be its sheaf of holomorphic
functions.
When understood from the context, the subscript
$M$ is usually suppressed and $\calo_M$ is
simply written $\calo$.
A sheaf of $\calo$-modules on a complex manifold is also called
an \term{analytic sheaf}.\index{analytic sheaf}

\begin{exa}
On a complex manifold $M$ of complex dimension $n$,
the sheaf $\Omega^k$ of holomorphic $k$-forms is an
analytic sheaf.
It is locally free of rank $\binom{n}{k}$, with local frame
$\{ dz_{i_1} \wedge\, \cdots \,\wedge dz_{i_k}\}$ for $1 \le i_1 < \cdots < i_k
\le n$.
\end{exa}

\begin{exa}
The sheaf $\calo^*$ of nowhere-vanishing holomorphic functions with pointwise multiplication on a
complex manifold $M$ is \emph{not} an analytic sheaf, since
multiplying a nowhere-vanishing function $f\in \calo^*(U)$ by the zero
function $0 \in \calo(U)$ will result in a function not in $\calo^*(U)$.
\end{exa}

Let $\calr$ be a sheaf of commutative rings on a topological
space $X$, let $\calf$ be a sheaf of $\calr$-modules on $X$,
and let
$f_1, \ldots, f_n$ be sections of $\calf$ over an open set $U$ in $X$.
For any $r_1, \ldots, r_n\in \calr(U)$, the map
\begin{align*}
\calr^{\oplus n} (U) &\to \calf(U),\\
(r_1, \ldots, r_n) &\mapsto \sum r_i f_i
\end{align*}
defines a sheaf map $\varphi\colon \calr^{\oplus n}|_U \to \calf|_U$ over
$U$.
The kernel of $\varphi$ is a subsheaf of $(\calr|_U)^{\oplus n}$
called the \term{sheaf of relations}\index{sheaf of relations} among $f_1, \ldots, f_n$,
denoted by $\cals(f_1, \ldots, f_n)$.
We say that $\calf|_U$ is \term{generated by} $f_1, \ldots, f_n$
if $\varphi\colon \calr^{\oplus n} \to \calf$ is surjective over $U$.\index{generators of a sheaf}

A sheaf $\calf$ of $\calr$-modules on $X$ is said to be
\term{of finite type}\index{sheaf of modules!of finite type} if every $x\in X$ has a neighborhood $U$ on
which $\calf$ is generated by finitely many sections $f_1, \ldots, f_n
\in \calf(U)$.
In particular, then, for every $y\in U$, the values $f_{1,y}, \dotsc,
f_{n,y} \in \calf_y$ generate the stalk $\calf_y$ as an
$\calr_y$-module.

\begin{defi}
A sheaf $\calf$ of $\calr$-modules on a topological space $X$ is
\term{coherent} if\index{coherent sheaf}
\begin{enumerate}[(i)]
\item  $\calf$ is of finite type; and
\item  for any open set $U \subset X$ and any collection of
  sections
$f_1, \ldots, f_n \in \calf(U)$, the sheaf $\cals(f_1, \ldots, f_n)$
of relations among $f_1, \ldots, f_n$ is of finite type over $U$.
\end{enumerate}
\end{defi}

\begin{theorem} \label{t:coherent}
\begin{enumerate}[(i)]
\item  The direct sum of finitely many coherent sheaves is
  coherent.
\item  The kernel, image, and cokernel of a morphism
of coherent sheaves are coherent.
\end{enumerate}
\end{theorem}

\pf
For a proof, see Serre \cite[Subsec.~13, Ths.~1 and 2, pp.~208--209]{serre55}.
 \epf

A sheaf $\calf$ of $\calr$-modules on a topological space $X$
is said to be \term{locally finitely presented}%
\index{locally finitely presented sheaf}\index{sheaf of modules!locally finitely presented} if every $x \in X$
has a neighborhood $U$ on which there is an
exact sequence of the form
\[
\calr|_U^{\ \ \oplus q} \to \calr|_U^{\ \ \oplus p} \to \calf|_U \to 0;
\]
in this case, we say that $\calf$ has a \term{finite presentation}\index{finite presentation!of a sheaf of modules}
or that $\calf$ is \term{finitely presented}\index{sheaf of modules!finitely presented} on $U$.
If $\calf$ is a coherent sheaf of $\calr$-modules on $X$, then
it is locally finitely presented.

\begin{remark}
Having a finite presentation locally is a consequence of coherence,
but is not equivalent to it.
Having a finite presentation means that for \emph{one set} of
generators of $\calf$, the sheaf of relations among them is finitely
generated.
Coherence is a stronger condition in that it requires the sheaf
of relations among \emph{any set} of elements of $\calf$ to be
finitely generated.
\end{remark}

A sheaf $\calr$ of rings on $X$ is clearly a sheaf
of $\calr$-modules of finite type.
For it to be coherent, for any open set $U \subset X$
and any sections $f_1, \ldots, f_n$, the sheaf
$\cals(f_1, \ldots, f_n)$ of relations among $f_1, \ldots, f_n$
must be of finite type.

\begin{exa} \label{exam:oka}
If $\calo_M$ is the sheaf of holomorphic functions on a complex manifold $M$,
then $\calo_M$ is a coherent sheaf of $\calo_M$-modules  (Oka's theorem
\cite[Sec.~5]{cartan}).\index{Oka's theorem}
\end{exa}

\begin{exa} \label{exam:serre}
If $\calo_X$ is the sheaf of regular functions on an algebraic variety $X$,
then $\calo_X$ is a coherent sheaf of $\calo_X$-modules
(Serre \cite[Sec.~37, Prop.~1]{serre55}).\index{Serre's theorem!on the coherence of the sheaf of regular functions}
\end{exa}

A sheaf of $\calo_X$-modules on an algebraic variety is called an
\term{algebraic sheaf}.\index{algebraic sheaf}

\begin{exa}
On a smooth variety $X$ of dimension $n$, the sheaf $\Omega^k$
of algebraic $k$-forms is an algebraic sheaf.  It is
locally free of rank $\binom{n}{k}$ \cite[Ch.~III, Th.~2, p.~200]{shafarevich}.

\end{exa}

\begin{theorem}
Let $\calr$ be a coherent sheaf of rings on a topological space $X$.
Then a sheaf $\calf$ of $\calr$-modules on $X$ is coherent
if and only if it is locally finitely presented.
\end{theorem}

\pf
$(\Rightarrow)$ True for any coherent sheaf of $\calr$-modules,
whether $\calr$ is coherent or not.

\noindent
$(\Leftarrow)$  Suppose there is an exact sequence
\[
\calr^{\oplus q} \to \calr^{\oplus p} \to \calf \to 0
\]
on an open set $U$ in $X$.
Since $\calr$ is coherent, by Theorem~\ref{t:coherent}
so are $\calr^{\oplus p}$, $\calr^{\oplus q}$, and the cokernel $\calf$ of
$\calr^{\oplus q} \to \calr^{\oplus p}$.
 \epf

Since the structure sheaves $\calo_X$ or $\calo_M$
of an algebraic variety $X$ or of a complex manifold $M$
are coherent, an algebraic or analytic sheaf is coherent
if and only if it is locally finitely presented.

\begin{exa}
A locally free analytic sheaf $\calf$ over a complex manifold $M$
is automatically coherent since every point $p$ has a
neighborhood $U$ on which there is an exact sequence of the form
\[
0 \to \calo_U^{\oplus p} \to \calf|_U \to 0,
\]
so that $\calf|_U$ is finitely presented.
\end{exa}

For our purposes, we define a \term{Stein manifold}\index{Stein manifold} to be a complex manifold that is
biholomorphic to a closed submanifold of $\C^N$
(this is not the usual definition, but is equivalent to it
\cite[p.~114]{hormander}).
In particular, a complex submanifold of $\C^N$ defined by
finitely many holomorphic functions is a Stein manifold.
One of the basic theorems about coherent analytic sheaves is Cartan's
theorem~B.

\begin{theorem}[Cartan's theorem B]
A coherent analytic sheaf $\calf$ is acyclic on a Stein manifold $M$,
i.e., $H^q(M, \calf) =0$ for all $q \ge 1$.\index{Cartan's theorem~B}
\end{theorem}

For a proof, see \cite[Th.~14, p.~243]{gunning--rossi}.

Let $X$ be a smooth quasi-projective variety defined over the complex
numbers and endowed with the Zariski topology.
The underlying set of $X$ with the complex topology is a complex
manifold $\Xand$.
Similarly, if $U$ is a Zariski open subset of $X$, let $U_{\an}$
be the underlying set of $U$ with the complex topology.
Since Zariski open sets are open in the complex topology,
$U_{\an}$ is open in  $\Xand$.

Denote by $\calo_{X\suban}$ the sheaf of holomorphic functions on
$\Xand$.
If $\calf$ is a coherent algebraic sheaf on $X$, then $X$ has an open
cover $\{ U \}$ by Zariski open sets such that on each open set $U$
there is an exact sequence
\[
\calo_U^{\oplus q} \to \calo_U^{\oplus p} \to \calf|_U \to 0
\]
of algebraic sheaves.
Moreover, $\{ U\suban \}$ is an open cover of $\Xand$ and
the morphism $\calo_U^{\oplus q} \to \calo_U^{\oplus p}$ of
algebraic sheaves induces a morphism
$\calo_{U\suban}^{\oplus q} \to \calo_{U\suban}^{\oplus p}$ of analytic sheaves.
Hence, there is a coherent analytic sheaf $\calf_{\an}$ on
the complex manifold $\Xand$ defined by
\[
\calo_{U\suban}^{\oplus q} \to \calo_{U\suban}^{\oplus p} \to \calf\suban|_{U\suban} \to 0.
\]
(Rename the open cover $\{ U\suban \}$ as $\{ U_{\ga} \}_{\ga \in
{\text{A}}}$.
A section of $\calf\suban$ over an open set $V \subset X\suban$ is a
collection of sections $s_{\ga} \in (\calf\suban|_{U_{\ga}})
(U_{\ga}\cap V)$
that agree on all pairwise intersections $(U_{\ga}\cap V) \cap (U_{\gb}\cap V)$.)

In this way one obtains a functor $(\ )\suban$ from the category
of smooth complex quasi-projective varieties and coherent algebraic
sheaves to the category of complex manifolds and analytic
sheaves.
Serre's GAGA (``G\'eom\'etrie alg\'ebrique et g\'eom\'etrie
analytique'') principle\index{GAGA principle} \cite{serre56} asserts that for smooth complex projective
varieties,
the functor $(\ )\suban$ is an equivalence of categories and moreover,
for all $q$, there are isomorphisms of cohomology groups \label{3p:gaga}
\begin{equation} \label{e:gaga}
H^q(X, \calf) \simeq H^q (\Xand, \calf\suban),
\end{equation}
where the left-hand side is the sheaf cohomology of $\calf$ on $X$ endowed
with the Zariski topology and the right-hand side is the sheaf cohomology
of $\calf\suban$ on $\Xand$ endowed with the complex topology.

When $X$ is a smooth complex quasi-projective variety,
to distinguish between sheaves of algebraic and sheaves of holomorphic forms,
we write $\Omega_{\alg}^p$
for the sheaf of algebraic $p$-forms on $X$
and $\Omega\suban^p$ for the sheaf of holomorphic $p$-forms on
$\Xand$
(for the definition of
algebraic forms, see the introduction to this chapter).
If $z_1, \ldots, z_n$ are local parameters for $X$ \cite[Ch.~II, Sec.~2.1, p.~98]{shafarevich},
then both $\Omega_{\alg}^p$ and $\Omega\suban^p$ are locally free with
frame $\{ dz_{i_1} \wedge \cdots \wedge dz_{i_p} \}$, where $I = (i_1,
\ldots, i_p)$ is a strictly increasing multiindex between $1$ and $n$
inclusive.  (For the algebraic case, see \cite[Vol.~1, Ch.~III, Sec.~5.4, Th.~4,
p.~203]{shafarevich}.)
Hence, locally there are sheaf isomorphisms
\[
0 \to \calo_U^{\binom{n}{p}} \to  \Omega_{\alg}^p|_U \to 0 \quad \text{and}
\quad 0 \to \calo_{U\suban}^{\binom{n}{p}} \to  \Omega_{\an}^p|_{U\suban}\to 0,
\]
which show that $\Omega\suban^p$ is the coherent analytic sheaf
associated to the coherent algebraic sheaf $\Omega_{\alg}^p$.

Let $k$ be a field. An \term{affine closed set} in $k^N$ is the zero
set of finitely many polynomials on $k^N$, and an \term{affine
  variety} is an algebraic variety biregular to an affine
closed set.
The algebraic analogue of Cartan's theorem~B is the following vanishing theorem
of Serre for an affine variety \cite[Sec.~44, Cor.~1, p.~237]{serre55}.

\begin{theorem}[Serre] \label{t:affinity}
A coherent algebraic sheaf $\calf$ on an affine variety $X$ is
acyclic on $X$,
i.e., $H^q(X, \calf) = 0$ for all $q \ge 1$.\index{Serre's theorem!on the cohomology of an affine\\ variety}
\end{theorem}

\section{The Hypercohomology of a Complex of Sheaves}\label{ch2:sec4}

This section requires some knowledge of double complexes
and their associated spectral sequences.
One possible reference is \cite[Chs.~2 and 3]{bott--tu}.
The hypercohomology of a complex $\call^{\bullet}$ of
sheaves of abelian groups on a topological space $X$
generalizes the cohomology of a single sheaf.
To define it, first form the double complex of
global sections of the Godement resolutions of the sheaves $\call^q$:
\[
K  = \bigoplus_{p,q} K^{p,q} = \bigoplus_{p,q} \Gamma(X, \calc^p\call^q).
\]
This double complex comes with two
differentials: a horizontal differential
\[
\delta\colon K^{p,q} \to K^{p+1, q}
\]
induced from the Godement resolution and a
vertical differential
\[
d\colon K^{p,q} \to K^{p,q+1}
\]
induced from the complex $\call^{\bullet}$.
Since
the differential $d\colon \call^q \to \call^{q+1}$
induces a morphism of complexes $\calc\bu \call^q
\to \calc\bu \call^{q+1}$,
where $\calc\bu$ is the Godement resolution,
the vertical differential in
the double complex $K$
commutes with the horizontal differential.
The \term{hypercohomology}\index{hypercohomology} $\hh^*(X, \call^{\bullet})$
of the complex $\call^{\bullet}$ is the total cohomology
of the double complex, i.e., the cohomology of the
associated single complex
\[
K\bu = \bigoplus K^k = \bigoplus_k \bigoplus_{p+q=k} K^{p,q}
\]
with differential $D = \delta + (-1)^p d$:
\[
\hh^k(X, \call^{\bullet}) = H_D^k (K\bu).
\]

If the complex of sheaves $\call^{\bullet}$ consists of a single
sheaf $\call^0 = \calf$ in degree~$0$,
\[
0 \to \calf \to 0 \to 0 \to \cdots,
\]
then the double complex $\bigoplus K^{p,q} = \bigoplus \Gamma(X, \calc^p \call^q)$ has nonzero entries only in the zeroth row, which
is simply the complex of sections of the Godement resolution
of $\calf$:
\begin{center}
\begin{pspicture}(-1,-.5)(7,2)
\psline{->}(0,-.1)(7,-.1)
\psline{->}(0,-.1)(0,1.7)
\psline(2,-.1)(2,1.7)
\psline(4,-.1)(4,1.7)
\psline(6,-.1)(6,1.7)
\uput{.2}[180](-0,.75){$K = \quad$}
\rput(1,-0.4){$0$}
\rput(3,-0.4){$1$}
\rput(5,-0.4){$2$}
\rput(1,0.25){$\Gamma(X, \calc^0\calf)$}
\rput(3,0.25){$\Gamma(X, \calc^1\calf)$}
\rput(5,0.25){$\Gamma(X, \calc^2\calf)$}
\rput(1,.75){$0$}
\rput(3,.75){$0$}
\rput(5,.75){$0$}
\rput(1,1.25){$0$}
\rput(3,1.25){$0$}
\rput(5,1.25){$0$}
\uput{.2}[270](7,-.1){$p$}
\uput{.2}[180](0,1.7){$q$}
\end{pspicture}
\end{center}
In this case, the associated single complex is the complex
$\Gamma(X, \calc^{\bullet} \calf)$ of global sections
of the Godement resolution of $\calf$,
and the hypercohomology of $\call^{\bullet}$
is the sheaf cohomology of $\calf$:
\begin{equation} \label{e:single}
\hh^k(X, \call^{\bullet}) = h^k \big(\Gamma(X, \calc^{\bullet}\calf)\big) = H^k(X,\calf).
\end{equation}
It is in this sense that hypercohomology generalizes sheaf
cohomology.

\subsection{The Spectral Sequences of Hypercohomology}\index{hypercohomology!spectral sequence of}%
\index{spectral sequence!of hypercohomology}

Associated to any double complex $(K$, $d$, $\delta)$
with commuting differentials $d$ and $\delta$ are two
spectral sequences converging to the total cohomology
$H_D^*(K)$.
One spectral sequence starts with $E_1= H_d$
and $E_2 = H_{\delta}H_d$.\index{spectral sequence!first}\index{spectral sequence!usual}
By reversing the roles of $d$ and $\delta$,
we obtain a second spectral sequence with $E_1= H_{\delta}$
and $E_2 = H_d H_{\delta}$
(see \cite[Ch.~II]{bott--tu}).\index{spectral sequence!second}
By the \emph{usual} spectral sequence\index{usual spectral sequence| \textit{see} {first spectral sequence }}
 of a double complex, we will
mean the first spectral sequence, with the vertical differential $d$
as the initial differential.

In the category of groups, the $E_{\infty}$ term is the associated graded
group of the total cohomology $H_D^*(K)$ relative to a canonically defined filtration
and is not necessarily isomorphic to $H_D^*(K)$ because of the extension
phenomenon in group theory.

Fix a nonnegative integer $p$ and let $T = \Gamma\big(X, \calc^p(\ )\big)$
be the Godement sections functor that associates to a sheaf $\calf$ on a
topological space $X$ the group of sections
$\Gamma(X, \calc^p\calf)$ of the Godement sheaf $\calc^p \calf$.
Since $T$ is an exact functor
by Corollary~\ref{c:functor},
by Proposition~\ref{p:exact} it commutes with cohomology:
\begin{equation} \label{e:commute}
h^q \big(T(\call^{\bullet})\big) = T\big(\calh^q(\call^{\bullet})\big),
\end{equation}
where $\calh^q := \calh^q(\call\bu)$ is the $q$th cohomology sheaf
of the complex $\call\bu$ (see Section~\ref{ss:cohomsheaves}).
For the double complex $K = \bigoplus \Gamma(X, \calc^p \call^q)$,
the $E_1$ term of the first spectral sequence is
the cohomology of $K$ with respect to the vertical differential $d$.
Thus, $E_1^{p,q} = H_d^{p,q}$ is the $q$th cohomology of the
$p$th column $K^{p,\bullet} = \Gamma\big(X, \calc^p(\call\bu)\big)$ of $K$:
\begin{alignat*}{2}
E_1^{p,q} = H_d^{p,q} &= h^q(K^{p,\bullet})
= h^q \big(\Gamma(X, \calc^p \call^{\bullet})\big)\\
&= h^q \big(T(\call^{\bullet})\big) &\quad&\text{(definition of $T$)} \\
&= T \big(\calh^q(\call^{\bullet})\big) &\quad&\text{(by \eqref{e:commute})}\\
&= \Gamma(X, \calc^p \calh^q) &\quad&\text{(definition of $T$)}.
\end{alignat*}
Hence, the $E_2$ term of the first spectral sequence is\index{first spectral sequence!$E_2$ term}
\begin{equation} \label{e:e2}
E_2^{p,q} = H_{\delta}^{p,q}(E_1) = H_{\delta}^{p,q}H_d^{\bullet,\bullet} = h_{\delta}^p(H_d^{\bullet,q})
= h_{\delta}^p \big(
\Gamma(X, \calc\bu\calh^q)\big) =\boxed{H^p(X, \calh^q)}.
\end{equation}

Note that the $q$th row of the double complex
$\bigoplus K^{p,q} = \bigoplus \Gamma(X, \calc^p \call^q)$
calculates the sheaf cohomology of $\call^q$ on $X$.
Thus, the $E_1$ term of the second spectral sequence is
\begin{equation} \label{e:e1}
E_1^{p,q} = H_{\delta}^{p,q} = h_{\delta}^p(K^{\bullet,q})
= h_{\delta}^p\big( \Gamma(X, \calc\bu \call^q)\big)
= \boxed{H^p (X, \call^q)}
\end{equation}
and the $E_2$ term is\index{second spectral sequence!$E_2$ term}
\[
E_2^{p,q} = H_d^{p,q}(E_1) = H_d^{p,q} H_{\delta}^{\bullet,\bullet}
= h_d^q\big(H_{\delta}^{p,\bullet}\big) = h_d^q\big(H^p(X,\call\bu)\big).
\]

\begin{theorem} \label{t:qi}
A quasi-isomorphism $\calf^{\bullet} \to \calg^{\bullet}$ of
complexes of sheaves of abelian groups over a topological space $X$
(see p.~\pageref{p:quasi})
induces a canonical isomorphism in hypercohomology:\index{quasi-isomorphism!induces an isomorphism in cohomology}
\[
\hh^*(X, \calf^{\bullet}) \overset{\sim}{\to} \hh^*(X, \calg^{\bullet}).
\]
\end{theorem}

\begin{proof}
By the functoriality of the Godement sections functors, a morphism $\calf\bu \to \calg\bu$ of complexes
of sheaves induces a homomorphism $\Gamma(X, \calc^p\calf^q) \to \Gamma(X, \calc^p\calg^q)$
that commutes with the two differentials $d$ and $\delta$ and hence induces a homomorphism
$\hh^*(X, \calf\bu) \to \hh^*(X,\calg\bu)$ in hypercohomology.

Since the spectral sequence construction is functorial, the morphism $\calf\bu \to \calg\bu$
also induces a morphism $E_r(\calf\bu) \to E_r(\calg\bu)$ of spectral sequences
and a morphism of the filtrations $$F_p\big( H_D(K_{\calf\bu})\big)\to F_p\big( H_D(K_{\calg\bu})\big)$$
on the hypercohomology
of $\calf\bu$ and $\calg\bu$.
We will shorten the notation $F_p\big( H_D(K_{\calf\bu})\big)$ to $F_p(\calf\bu)$.

By definition, the quasi-isomorphism $\calf^{\bullet} \to \calg^{\bullet}$
induces an isomorphism of cohomology sheaves $\calh^*(\calf^{\bullet})
\overset{\sim}{\to} \calh^*(\calg^{\bullet})$,
and by \eqref{e:e2} an isomorphism of the $E_2$ terms of the first
spectral sequences of $\calf^{\bullet}$ and of $\calg^{\bullet}$:
\[
E_2^{p,q}(\calf^{\bullet}) = H^p\big(X, \calh^q(\calf^{\bullet})\big) \overset{\sim}{\to}
H^p\big(X, \calh^q(\calg^{\bullet})\big) = E_2^{p,q}(\calg^{\bullet}).
\]
An isomorphism of the $E_2$ terms induces an isomorphism
of the $E_{\infty}$ terms:
\[
\bigoplus_p \frac{F_p(\calf\bu)}{F_{p+1}(\calf\bu)} = E_{\infty}(\calf\bu) \overset{\sim}{\to}
E_{\infty}(\calg\bu) = \bigoplus_p \frac{F_p(\calg\bu)}{F_{p+1}(\calg\bu)}.
\]

We claim that in fact, the canonical homomorphism $\hh^*(X,\calf\bu) \to \hh^*(X, \calg\bu)$ is
an isomorphism.
Fix a total degree $k$ and let $F_p^k(\calf\bu) = F_p(\calf\bu) \cap \hh^k(X,\calf\bu)$.
Since $$K^{\bullet,\bullet}(\calf\bu) = \bigoplus \Gamma(X, \calc^p\calf^q)$$ is a first-quadrant double complex,
the filtration $\{ F_p^k(\calf\bu) \}_p$ on $\hh^k(X,\calf\bu)$ is finite in length:
\[
\hh^k(X,\calf\bu) = F_0^k(\calf\bu) \supset F_1^k(\calf\bu) \supset \cdots \supset F_{k}^k(\calf\bu) \supset
F_{k+1}^k(\calf\bu) = 0.
\]
A similar finite filtration $\{F_p^k(\calg\bu)\}_p$ exists on $\hh^k(X, \calg\bu)$.

%
%
Suppose $F_p^k(\calf\bu) \to F_p^k(\calg\bu)$ is an isomorphism.
We will prove that $F_{p-1}^k(\calf\bu) \to F_{p-1}^k(\calg\bu)$ is an isomorphism.
In the commutative diagram
\[
\xymatrix{
0 \ar[r] & F_{p}^k(\calf\bu) \ar[r] \ar[d] & F_{p-1}^k(\calf\bu) \ar[r] \ar[d] & F_{p-1}^k(\calf\bu)/F_{p}^k(\calf\bu) \ar[r] \ar[d] & 0 \\
0 \ar[r] & F_{p}^k(\calg\bu) \ar[r] & F_{p-1}^k(\calg\bu) \ar[r]  & F_{p-1}^k(\calg\bu)/F_{p}^k(\calg\bu) \ar[r] & 0,
}
\]
the two outside vertical maps are isomorphisms,
by the induction hypothesis and because $\calf\bu \to \calg\bu$ induces an
isomorphism of the associated graded groups.
By the five lemma, the middle vertical map $F_{p-1}^k(\calf\bu) \to F_{p-1}^k(\calg\bu)$ is also
an isomorphism.
By induction on the filtration subscript $p$, as $p$ moves from $k+1$ to $0$,
we conclude that
\[
\hh^k(X, \calf\bu) = F_{0}^k(\calf\bu)\to  F_{0}^k(\calg\bu) = \hh^k(X, \calg\bu)
\]
is an isomorphism.
\end{proof}

\begin{theorem} \label{t:acyclic}
If $\call^{\bullet}$ is a complex of acyclic sheaves of abelian groups on
a topological space $X$, then the hypercohomology of $\call\bu$
is isomorphic to the cohomology of the complex of
global sections of $\call\bu$:\index{hypercohomology!of a complex of acyclic sheaves}
\[
\hh^k(X, \call\bu) \simeq h^k\big(\call\bu(X)\big),
\]
where $\call\bu(X)$ denotes the complex
\[
0 \to \call^0(X) \to \call^1(X) \to \call^2(X) \to \cdots .
\]
\end{theorem}

\pf
Let $K$ be the double complex $K = \bigoplus K^{p,q} = \bigoplus
\calc^p \call^q(X)$.
Because each $\call^q$ is acyclic on $X$,
in the second spectral sequence of $K$,
by \eqref{e:e1} the $E_1$ term is
\[
E_1^{p,q}   = H^p(X, \call^q) =
\begin{cases}
 \call^q(X) &\text{for } p = 0,\\
0 &\text{for } p > 0.
\end{cases}
\]
\begin{center}
\begin{pspicture}(-1,-.5)(7,2)
\psline{->}(0,-.1)(7,-.1)
\psline{->}(0,-.1)(0,1.7)
\psline(2,-.1)(2,1.7)
\psline(4,-.1)(4,1.7)
\psline(6,-.1)(6,1.7)
\uput{.2}[180](-0,.75){$E_1=H_{\delta} =\quad$}
\rput(1,-0.4){$0$}
\rput(3,-0.4){$1$}
\rput(5,-0.4){$2$}
\rput(1,0.25){$\call^0(X)$}
\rput(3,0.25){$0$}
\rput(5,0.25){$0$}
\rput(1,.75){$\call^1(X)$}
\rput(3,.75){$0$}
\rput(5,.75){$0$}
\rput(1,1.25){$\call^2(X)$}
\rput(3,1.25){$0$}
\rput(5,1.25){$0$}
\uput{.2}[270](7,-.1){$p$}
\uput{.2}[180](0,1.7){$q$}
\end{pspicture}
\end{center}
Hence,
\[
E_2^{p,q} = H_d^{p,q}H_{\delta} =
\begin{cases}
\ h^q \big(\call\bu(X)\big) &\text{for } p=0,\\
\ 0 &\text{for } p > 0.
\end{cases}
\]
Therefore, the spectral sequence degenerates at the $E_2$ term
and
\[
\hh^k(X, \call\bu) \simeq E_2^{0,k} = h^k\big(\call\bu(X)\big).
\]
\end{proof}

\subsection{Acyclic Resolutions}

Let $\calf$ be a sheaf of abelian groups on a topological space $X$.
A resolution
\[
0 \to \calf \to \call^0 \to \call^1 \to \call^2 \to \cdots
\]
of $\calf$ is said to be \term{acyclic}\index{acyclic resolution}\index{resolution!acyclic} on $X$ if each sheaf $\call^q$
is acyclic on $X$, i.e., $H^k(X, \call^q) = 0$ for all $k >0$.

If $\calf$ is a sheaf on $X$, we will denote by $\calf\bu$ the
complex of sheaves such that $\calf^0 = \calf$ and $\calf^k =0$ for
$k >0$.

\begin{theorem} \label{4t:acyclic}
If $\ 0 \to \calf \to \call\bu$ is an acyclic resolution of the sheaf $\calf$
on
a topological space $X$, then the cohomology of $\calf$ can be
computed from the complex of global sections of $\call\bu$:\index{sheaf cohomology!using acyclic resolution}
\[
H^k(X,\calf) \simeq h^k \big(\call\bu(X)\big).
\]
\end{theorem}

\pf
The resolution $0 \to \calf \to \call\bu$ may be viewed as a
quasi-isomorphism of the two complexes
\[
\xymatrix{
0 \ar[r] & \calf \ar[d] \ar[r] & 0 \ar[d] \ar[r] & 0 \ar[d] \ar[r] &
\cdots\\
0 \ar[r] & \call^0 \ar[r] & \call^1 \ar[r] & \call^2 \ar[r]  & \cdots ,
}
\]
since
\[
\calh^0(\mbox{top row}) = \calh^0(\calf\bu) = \calf
\simeq \im(\calf \to \call^0)
= \ker(\call^0 \to \call^1) =
\calh^0(\mbox{bottom row})
\]
and the higher cohomology sheaves of both complexes are zero.
By Theorem~\ref{t:qi}, there is an induced morphism
in hypercohomology
\[
\hh^k (X, \calf\bu) \simeq \hh^k(X, \call\bu).
\]
The left-hand side is simply the sheaf cohomology $H^k(X, \calf)$
by \eqref{e:single}.
By Theorem~\ref{t:acyclic}, the right-hand side is
$h^k(\call\bu(X))$.
Hence,
\[
H^k(X,\calf) \simeq h^k \big(\call\bu(X)\big). 
\]
\epf
So in computing
sheaf cohomology, any acyclic resolution of $\calf$ on a topological space $X$
can take the place of the Godement resolution.

Using acyclic resolutions,
we can give simple proofs of
de Rham's and Dolbeault's
theorems.

\begin{exa}[De Rham's theorem]\index{de Rham theorem}
By the Poincar\'e lemma (\cite[Sec.~4, p.~33]{bott--tu},
\cite[p.~38]{griffiths--harris}),\index{Poincar\'e lemma} on a
$\cinf$ manifold $M$ the sequence of sheaves
\begin{equation} \label{4e:deRham}
0 \to \underline{\R} \to \cala^0 \to \cala^1 \to \cala^2 \to \cdots
\end{equation}
is exact.
Since each $\cala^k$ is fine and hence acyclic on $M$,
\eqref{4e:deRham} is an acyclic resolution of $\underline{\R}$.
By Theorem~\ref{4t:acyclic},
\[
H^*(M, \underline{\R}) \simeq h^*\big(\cala\bu(M)\big) = H_{\text{dR}}^*(M).
\]
Because the sheaf cohomology $H^*(M, \underline{\R})$ of a manifold
is isomorphic to the real singular cohomology of $M$
(Remark~\ref{r:cohomology}),
de Rham's theorem follows.

\end{exa}

\begin{exa}[Dolbeault's theorem]\index{Dolbeault theorem}
According to the $\bar{\partial}$-Poincar\'e lemma \cite[pp.~25 and 38]{griffiths--harris},
on a complex manifold $M$
the sequence of sheaves
\[
0 \to \Omega^p \to \cala^{p,0} \overset{\bar{\partial}}{\to}
\cala^{p,1} \overset{\bar{\partial}}{\to}
\cala^{p,2} \to \cdots
\]
is exact.
As in the previous example, because each sheaf $\cala^{p,q}$ is
fine and hence acyclic, by Theorem~\ref{4t:acyclic},
\[
H^q(M, \Omega^p) \simeq h^q\big(\cala^{p,\bullet}(M)\big)
= H^{p,q}(M).
\]
This is the Dolbeault isomorphism for a complex
manifold $M$.
\end{exa}

\section{The Analytic de Rham Theorem}

The analytic de Rham theorem is the analogue of the classical de Rham
theorem for a complex manifold,
according to which the singular cohomology with $\C$ coefficients of any
complex manifold can be computed from its sheaves of holomorphic
forms.
Because of the holomorphic Poincar\'e lemma,
the analytic de Rham theorem is far easier to prove than its
algebraic counterpart.

\subsection{The Holomorphic Poincar\'e Lemma}

Let $M$ be a complex manifold and $\Omega\suban^k$ the sheaf of holomorphic
$k$-forms on $M$.
Locally, in terms of complex coordinates $z_1, \ldots, z_n$, a
holomorphic form can be written as $\sum a_I\, dz_{i_1} \wedge \cdots
\wedge dz_{i_n}$, where the $a_I$ are holomorphic functions.
Since for a holomorphic function $a_I$,
\[
d a_I = \partial a_I + \bar{\partial} a_I = \sum_i \dpd{a_I}{z_i}\, dz_i + \sum_i
\dpd{a_I}{\bar{z}_i} d\bar{z}_i
= \sum_i \dpd{a_I}{z_i} dz_i,
\]
the exterior derivative $d$ maps holomorphic forms to
holomorphic forms.
Note that $a_I$ is holomorphic if and only if $\bar{\partial} a_I = 0$.

\begin{theorem}[Holomorphic Poincar\'e lemma]\label{t:holompoincare}
On a complex manifold $M$ of complex dimension $n$, the sequence\index{holomorphic Poincar\'e lemma}%
\index{Poincar\'e lemma!holomorphic}
\[
0 \to \underline{\C} \to \Omega\suban^0 \overset{d}{\to} \Omega\suban^1
\overset{d}{\to} \cdots \to \Omega\suban^n \to 0
\]
of sheaves is exact.
\end{theorem}

\pf
We will deduce the holomorphic Poincar\'e lemma from the smooth
Poincar\'e lemma and the $\bar{\partial}$-Poincar\'e lemma
by a double complex argument.
The double complex $\bigoplus \cala^{p,q}$ of sheaves of smooth
$(p,q)$-forms has two differentials $\partial$ and $\bar{\partial}$.
These differentials anticommute because
\begin{align*}
0 &= d\comp d = (\partial + \bar{\partial}) (\partial +
\bar{\partial})
= \partial^2 + \bar{\partial}  \partial + \partial
\bar{\partial}
+\bar{\partial}^2 \\
& = \bar{\partial} \partial + \partial \bar{\partial}.
\end{align*}
The associated single complex $\bigoplus \cala_{\C}^k$,
where $\cala_{\C}^k = \bigoplus_{p+q=k} \cala^{p,q}$ with
differential
$d = \partial + \bar{\partial}$, is simply the usual complex of sheaves
of smooth $\C$-valued differential forms on $M$.
By the smooth Poincar\'e lemma,
\[
\calh_d^k (\cala_{\C}^{\bullet}) = \begin{cases}
\underline{\C} &\text{for } k=0,\\
0 &\text{for } k > 0.
\end{cases}
\]

By the $\bar{\partial}$-Poincar\'e lemma, the sequence
\[
0 \to \Omega\suban^p \to \cala^{p,0} \overset{\bar{\partial}}{\to}
\cala^{p,1} \overset{\bar{\partial}}{\to} \cdots \to
\cala^{p,n} \to 0
\]
is exact for each $p$ and so the $E_1$ term of the
usual spectral sequence of the
double complex $\bigoplus \cala^{p,q}$ is
\begin{center}
\begin{pspicture}(-1,-.5)(7,2)
\psline{->}(0,-.1)(7,-.1)
\psline{->}(0,-.1)(0,1.7)
\psline(2,-.1)(2,1.7)
\psline(4,-.1)(4,1.7)
\psline(6,-.1)(6,1.7)
\uput{.2}[180](-0,.75){$E_1 = H_{\bar{\partial}} =$}
\rput(1,-0.4){$0$}
\rput(3,-0.4){$1$}
\rput(5,-0.4){$2$}
\rput(1,0.25){$\Omega\suban^0$}
\rput(3,0.25){$\Omega\suban^1$}
\rput(5,0.25){$\Omega\suban^2$}
\rput(1,.75){$0$}
\rput(3,.75){$0$}
\rput(5,.75){$0$}
\rput(1,1.25){$0$}
\rput(3,1.25){$0$}
\rput(5,1.25){$0$}
\uput{.2}[270](7,-.1){$p$}
\uput{.2}[180](0,1.7){$q$}
\rput(7.2,-.1){.}
\end{pspicture}
\end{center}
Hence, the $E_2$ term is given by
\[
E_2^{p,q} =
\begin{cases}
\calh_d^p(\Omega\suban\bu)&\text{for } q = 0,\\
  0 &\text{for }q > 0.
\end{cases}
\]
Since the spectral sequence degenerates at the $E_2$ term,
\[
\calh_d^k(\Omega\suban\bu) = E_2 = E_{\infty} \simeq \calh_d^k(\cala_{\C}\bu) =
\begin{cases}
\underline{\C} &\text{for } k=0,\\
0 &\text{for } k > 0,
\end{cases}
\]
which is precisely the holomorphic Poincar\'e lemma.
 \epf

\subsection{The Analytic de Rham Theorem}\index{analytic de Rham theorem}\index{de Rham theorem!analytic}

\begin{theorem}\label{t:analytic}
Let $\Omega\suban^k$ be the sheaf of holomorphic $k$-forms on a complex
manifold $M$.
Then the singular cohomology of $M$ with complex coefficients can be
computed as the hypercohomology of the complex $\Omega\suban\bu$:
\[
H^k(M,\C) \simeq \hh^k (M, \Omega\suban\bu).
\]
\end{theorem}

\pf
Let $\underline{\C}\bu$ be the complex of sheaves that is
$\underline{\C}$ in degree $0$ and zero otherwise.
The holomorphic Poincar\'e lemma may be interpreted as
a quasi-isomorphism of the two complexes
\[
\xymatrix{
0 \ar[r] &\underline{\C} \ar[d] \ar[r] & 0\ar[d] \ar[r] & 0 \ar[d] \ar[r] & {\cdots} \\
0 \ar[r] &\Omega_{\an}^0 \ar[r] &
\Omega_{\an}^1 \ar[r] & \Omega_{\an}^2 \ar[r] & {\cdots} ,
}
\]
since
\begin{alignat*}{2}
 \calh^0(\underline{\C}\bu) &= \underline{\C}
\simeq \im (\underline{\C} \to \Omega\suban^0)\\
&=\ker(\Omega\suban^0 \to \Omega\suban^1) &\quad&\text{(by the holomorphic
Poincar\'e lemma)} \\
&= \calh^0(\Omega\suban\bu)
\end{alignat*}
and the higher cohomology sheaves of both complexes are zero.

By Theorem~\ref{t:qi}, the quasi-isomorphism $\underline{\C}\bu \simeq
\Omega\suban\bu$ induces an isomorphism
\begin{equation}\label{e:hyper}
\hh^*(M, \underline{\C}\bu) \simeq \hh^*(M, \Omega\suban\bu)
\end{equation}
in hypercohomology.
Since $\underline{\C}\bu$ is a complex of sheaves concentrated in
degree $0$,
by \eqref{e:single} the left-hand side of \eqref{e:hyper} is
the sheaf cohomology $H^k(M, \underline{\C})$,
which is isomorphic to the singular cohomology $H^k(M, \C)$
by Remark~\ref{r:cohomology}.
 \epf

In contrast to the sheaves $\cala^k$ and $\cala^{p,q}$ in de Rham's theorem and Dolbeault's theorem,
the sheaves $\Omega\bu\suban$ are generally neither fine nor acyclic,
because in the analytic category there is no partition of unity.
However,
when $M$ is a Stein manifold, the complex $\Omega\suban\bu$ is a complex
of acyclic sheaves on $M$ by Cartan's theorem~B.
It then follows from Theorem~\ref{t:acyclic} that
\[
\hh^k(M, \Omega\suban\bu) \simeq h^k \big(\Omega\suban\bu(M)\big).
\]
This proves the following corollary of Theorem~\ref{t:analytic}.

\begin{cor} \label{c:analytic}
The singular cohomology of a Stein manifold $M$ with coefficients in $\C$
can
be computed from the holomorphic de Rham complex:\index{singular cohomology!of a Stein manifold}
\[
H^k(M, \C) \simeq h^k \big(\Omega\suban\bu(M)\big).
\]
\end{cor}

\section{The Algebraic de Rham Theorem for a Projective Variety}

Let $X$ be a smooth complex algebraic variety with the
Zariski topology.
The underlying set of $X$ with the complex topology
is a complex manifold $\Xand$.
Let $\Omega_{\alg}^k$ be the sheaf of algebraic
$k$-forms on $X$,
and $\Omega_{\an}^k$ the sheaf
of holomorphic $k$-forms on $\Xand$.
According to the holomorphic Poincar\'e lemma
(Theorem~\ref{t:holompoincare}),
the complex of sheaves
\begin{equation} \label{e:holomorphic}
0 \to \underline{\C} \to \Omega_{\an}^0 \overset{d}{\to}
\Omega_{\an}^1 \overset{d}{\to}
\Omega_{\an}^2 \overset{d}{\to}  \cdots
\end{equation}
is exact.
However, there is no Poincar\'e lemma in the algebraic
category; the complex
\[
0 \to \underline{\C} \to \Omega_{\alg}^0 \to
\Omega_{\alg}^1 \to
\Omega_{\alg}^2 \to \cdots
\]
is in general not exact.

\begin{theorem}[Algebraic de Rham theorem for a projective variety]\label{t:projective}
If $X$ is a smooth complex projective variety, then there is
an isomorphism\index{algebraic de Rham theorem!for a projective variety}
\[
H^k(\Xand, \C) \simeq \hh^k(X, \Omega\bu\alg)
\]
between the singular cohomology of $\Xand$ with coefficients in $\C$
and the hypercohomology of $X$ with coefficients in the
complex $\Omega\bu\alg$ of sheaves of algebraic differential
forms on $X$.
\end{theorem}

\pf
By Theorem~\ref{t:qi}, the quasi-isomorphism $\underline{\C}\bu \to
\Omega_{\an}\bu$ of complexes of sheaves induces
an isomorphism in hypercohomology
\begin{equation} \label{e:1}
\hh^*(\Xand, \underline{\C}\bu) \simeq \hh^*(\Xand, \Omega_{\an}\bu).
\end{equation}
In the second spectral sequence converging to
$\hh^*(\Xand, \Omega_{\an}\bu)$,
by \eqref{e:e1} the $E_1$ term is
\[
E_{1, {\an}}^{p,q} = H^p (\Xand, \Omega_{\an}^q).
\]
By \eqref{e:e1} the $E_1$ term
in the second spectral sequence converging to the
hypercohomology $\hh^*(X, \Omega_{\alg}\bu)$ is
\[
E_{1,{\alg}}^{p,q} = H^p (X, \Omega_{\alg}^q).
\]

Since $X$ is a smooth complex projective variety,
Serre's GAGA principle \eqref{e:gaga} applies  and gives an isomorphism
\[
H^p (X, \Omega_{\alg}^q) \simeq H^p (\Xand, \Omega_{\an}^q).
\]
The isomorphism $E_{1,{\alg}} \overset{\sim}{\to}E_{1, {\an}}$
induces an isomorphism in $E_{\infty}$.
Hence,
\begin{equation} \label{e:2}
\hh^*(X, \Omega_{\alg}\bu) \simeq \hh^*(\Xand, \Omega_{\an}\bu).
\end{equation}
Combining \eqref{e:single}, \eqref{e:1}, and \eqref{e:2} gives
\[
H^*(\Xand, \underline{\C}) \simeq \hh^*(\Xand, \underline{\C}\bu)
\simeq \hh^*(\Xand, \Omega_{\an}\bu)
\simeq \hh^*(X, \Omega_{\alg}\bu). 
\]
Finally, by the isomorphism between sheaf cohomology
and singular cohomology (Remark~\ref{r:cohomology}), we may replace the
sheaf cohomology $H^*(\Xand, \underline{\C})$ by a singular cohomology group:
\[
H^*(\Xand, {\C}) \simeq \hh^*(X, \Omega_{\alg}\bu).
\]
\epf

\section*{Part II.~\v{C}ech Cohomology and the Algebraic de Rham Theorem in General}
\addcontentsline{toc}{section}{\textbf{Part II.  \v{C}ech Cohomology and the Algebraic de Rham Theorem in General}}

The algebraic de Rham theorem (Theorem~\ref{t:projective}) in fact
does not require the hypothesis of projectivity on $X$.
In this section we will extend it to an arbitrary smooth algebraic
variety defined over $\C$.
In order to carry out this extension, we will need to develop two more
machineries: the \v{C}ech cohomology of a sheaf and
the \v{C}ech cohomology of a complex of sheaves.
\v{C}ech cohomology provides a practical method for computing
sheaf cohomology and hypercohomology.

\section{\v{C}ech Cohomology of a Sheaf}

\v{C}ech cohomology may be viewed as a generalization of the Mayer--Vietoris sequence
from two open sets to arbitrarily many open sets.

\subsection{\v{C}ech Cohomology of an Open Cover}

Let $\fraku = \{ U_{\ga} \}_{\ga \in \Alpha}$ be an open cover of
the topological space $X$ indexed by a linearly ordered set
$\Alpha$, and $\calf$ a presheaf of abelian groups on $X$. To
simplify the notation, we will write the $(p+1)$-fold intersection
$U_{\ga_0} \cap \cdots \cap U_{\ga_p}$ as $U_{\ga_0 \cdots
\ga_p}$. Define the \term{group of \v{C}ech $p$-cochains}\index{Cech cochains@\v{C}ech cochains}
on $\fraku$ with values in the presheaf
$\calf$ to be the direct product
\[
\check{C}^p(\fraku, \calf) := \prod_{\ga_0 < \cdots < \ga_p} \calf(U_{\ga_0
\cdots \ga_p}).
\]
An element $\omega$ of $\check{C}^p(\fraku, \calf)$ is then a
function that assigns to each finite set of indices $\ga_0, \ldots, \ga_p$
an element $\omega_{\ap} \in \calf(U_{\ap})$. We will write
$\omega= (\omega_{\ap})$, where the subscripts range over all
$\ga_0 < \cdots < \ga_p$.
In particular, the subscripts $\ga_0, \ldots, \ga_p$ must all be distinct.
Define the \term{\v{C}ech coboundary
operator}\index{Cech coboundary operator@\v{C}ech coboundary operator}
\[
\delta = \delta_p\colon \check{C}^p(\fraku, \calf) \rightarrow
\check{C}^{p+1}(\fraku, \calf)
\]
by the alternating sum formula
\[
(\delta \omega)_{\apone} = \sum_{i=0}^{p+1} (-1)^i \omega_{\ga_0
\cdots \widehat{\ga_i} \cdots \ga_{p+1}},
\]
where 
$\widehat{\ga_i}$ means to omit
the index $\ga_i$; moreover,
 the restriction of $\omega_{\ga_0
\cdots \widehat{\ga_i} \cdots \ga_{p+1}}$ from $U_{\ga_0 \cdots
\widehat{\ga_i} \cdots \ga_{p+1}}$ to $U_{\apone}$ is suppressed
in the notation.

\begin{prop} \label{6p:d2} If $\delta$ is the \v{C}ech coboundary
operator, then $\delta^2 =0$. \end{prop}

\pf Basically, this is true because in $(\delta^2 \omega)_{\ga_0
\cdots \ga_{p+2}}$, we omit two indices $\ga_i$, $\ga_j$ twice
with opposite signs. To be precise,
\begin{align*}
(\delta^2 \omega)_{\ga_0  \cdots \ga_{p+2}} &=
\sum (-1)^i (\delta \omega)_{\ga_0 \cdots \widehat{\ga_i} \cdots \ga_{p+2}} \\
&= \sum_{j < i} (-1)^i (-1)^j
\omega_{\ga_0 \cdots\widehat{\ga_j} \cdots \widehat{\ga_i} \cdots \ga_{p+2}} \\
&\qquad + \sum_{j > i} (-1)^i (-1)^{j-1}
\omega_{\ga_0 \cdots\widehat{\ga_i} \cdots \widehat{\ga_j} \cdots \ga_{p+2}} \\
&= 0. 
\end{align*}
\epf

It follows from Proposition~\ref{6p:d2} that $\check{C}\bu(\fraku, \calf)
:= \bigoplus_{p=0}^{\infty} \check{C}^p(\fraku,\calf)$ is a cochain
complex with differential $\delta$. The
cohomology of the complex $(\check{C}^*(\fraku,\calf),\delta)$,
\[
\check{H}^p(\fraku,\calf) := \frac{\ker \delta_p}{\im \delta_{p-1}}
= \frac{\{ p\text{-cocycles}\}}{\{ p\text{-coboundaries}\}},
\]
is called the \term{\v{C}ech cohomology}%
\index{Cech cohomology@\v{C}ech cohomology}
 of the open cover $\fraku$ with values in the presheaf $\calf$.

\subsection{Relation Between \v{C}ech Cohomology and Sheaf Cohomology}

In this subsection we construct a natural map from the \v{C}ech
cohomology of a sheaf on an open cover to its sheaf cohomology.
This map is based on a property of flasque sheaves.

\begin{lemma} \label{l:flasque-cech}
Suppose $\calf$ is a flasque sheaf of abelian groups on a topological space $X$,
and $\fraku = \{ U_{\ga}\}$ is an open cover of $X$.
Then the augmented \v{C}ech complex
\[
0 \to \calf(X) \to \prod_{\ga} \calf(U_{\ga}) \to \prod_{\ga < \gb}
\calf(U_{\ga\gb})
\to \cdots
\]
is exact.
\end{lemma}
\noindent
In other words, for a flasque sheaf $\calf$ on $X$,
\[
\check{H}^k(\fraku, \calf) =
\begin{cases}
\ \calf(X) &\text{for } k=0,\\
\ 0 & \text{for } k > 0.
\end{cases}
\]

\pf
\cite[Th.~5.2.3(a), p.~207]{godement}.
 \epf

Now suppose $\calf$ is any sheaf of abelian groups on a topological
space $X$ and $\fraku = \{ U_{\ga}\}$ is an open cover of $X$.
Let $K^{\bullet,\bullet}= \bigoplus K^{p,q}$ be the double complex
\[
K^{p,q} = \check{C}^p(\fraku, \calc^q\calf) = \prod_{\ga_0 < \cdots < \ga_p}
\calc^q\calf(U_{\ga_0 \cdots \ga_p}).
\]
We augment this complex with an outside bottom row ($q=-1$) and
an outside left column ($p = -1$):

\begin{equation}\label{e:cg}
\vcenter\bgroup
\xy
(-30,-5)*{0};
(-10,-5)*{\calf(X)};
(20,-5)*{\prod\calf(U_{\ga})};
(50,-5)*{\prod\calf(U_{\ga\gb})};
(-30,10)*{0};
(-10,10)*{\calc^0\calf(X)};
(20,10)*{\prod\calc^0\calf(U_{\ga})};
(50,10)*{\prod\calc^0\calf(U_{\ga\gb})};
(-30,25)*{0};
(-10,25)*{\calc^1\calf(X)};
(20,25)*{\prod\calc^1\calf(U_{\ga})};
(50,25)*{\prod\calc^1\calf(U_{\ga\gb})};
{\ar (4,38)*{}; (4,3)*{};};
{\ar (80,3)*{}; (4,3)*{};};
(81,0)*{p}; (2,40)*{q};
{\ar (-20,-5)*{}; (-25,-5)*{};};
{\ar (-20,10)*{}; (-25,10)*{};};
{\ar (-20,25)*{}; (-25,25)*{};};
{\ar (7,-5)*{}; (2,-5)*{};};
{\ar (7,10)*{}; (2,10)*{};};
{\ar (7,25)*{}; (2,25)*{};};
{\ar (37,-5)*{}; (33,-5)*{};};
{\ar (37,10)*{}; (33,10)*{};};
{\ar (37,25)*{}; (33,25)*{};};
{\ar (67,-5)*{}; (63,-5)*{};};
{\ar (67,10)*{}; (63,10)*{};};
{\ar (67,25)*{}; (63,25)*{};};
{\ar (-10,5.5)*{}; (-10,0)*{};};
{\ar (20,5.5)*{}; (20,0)*{};};
{\ar (50,5.5)*{}; (50,0)*{};};
(22,1.5)*{\epsilon}; (52,1.5)*{\epsilon};
{\ar (-10,19.5)*{}; (-10,15.5)*{};};
{\ar (20,19.5)*{}; (20,15.5)*{};};
{\ar (50,19.5)*{}; (50,15.5)*{};};
{\ar (-10,34.5)*{}; (-10,30.5)*{};};
{\ar (20,34.5)*{}; (20,30.5)*{};};
{\ar (50,34.5)*{}; (50,30.5)*{};};
\endxy
\egroup
\end{equation}

\bigskip

Note that the $q$th row of the double complex $K^{\bullet,\bullet}$ is the
\v{C}ech cochain complex of the Godement sheaf $\calc^q\calf$
and the $p$th column is the complex of groups for computing the sheaf
cohomology $\prod_{\ga_0 < \cdots < \ga_p}
H^*(U_{\ga_0 \cdots \ga_p}, \calf)$.

By Lemma~\ref{l:flasque-cech}, each row of the augmented double complex
\eqref{e:cg} is exact. Hence, the $E_1$ term of the second spectral sequence
of the double complex is
\begin{center}
\begin{pspicture}(-1,-.5)(7,2)
\psline{->}(0,-.1)(7,-.1)
\psline{->}(0,-.1)(0,1.7)
\psline(2,-.1)(2,1.7)
\psline(4,-.1)(4,1.7)
\psline(6,-.1)(6,1.7)
\uput{.2}[180](-0,.75){$E_1 = H_{\delta} =$}
\rput(1,-0.4){$0$}
\rput(3,-0.4){$1$}
\rput(5,-0.4){$2$}
\rput(1,0.25){$\calc^0\calf(X)$}
\rput(3,0.25){0}
\rput(5,0.25){0}
\rput(1,.75){$\calc^1\calf(X)$}
\rput(3,.75){$0$}
\rput(5,.75){$0$}
\rput(1,1.25){$\calc^2\calf(X)$}
\rput(3,1.25){$0$}
\rput(5,1.25){$0$}
\uput{.2}[270](7,-.1){$p$}
\uput{.2}[180](0,1.7){$q$}
\end{pspicture}
\end{center}
and the $E_2$ term is
\begin{center}
\begin{pspicture}(-1,-.5)(7,2)
\psline{->}(0,-.1)(7,-.1)
\psline{->}(0,-.1)(0,1.7)
\psline(2,-.1)(2,1.7)
\psline(4,-.1)(4,1.7)
\psline(6,-.1)(6,1.7)
\uput{.2}[180](-0,.75){$E_2 = H_d H_{\delta} =$}
\rput(1,-0.4){$0$}
\rput(3,-0.4){$1$}
\rput(5,-0.4){$2$}
\rput(1,0.25){$H^0(X, \calf)$}
\rput(3,0.25){0}
\rput(5,0.25){0}
\rput(1,.75){$H^1(X, \calf)$}
\rput(3,.75){$0$}
\rput(5,.75){$0$}
\rput(1,1.25){$H^2(X, \calf)$}
\rput(3,1.25){$0$}
\rput(5,1.25){$0$}
\uput{.2}[270](7,-.1){$p$}
\uput{.2}[180](0,1.7){$q$}
\rput(7.3,-.1){.}
\end{pspicture}
\end{center}
So the second spectral sequence of the double complex \eqref{e:cg} degenerates at the $E_2$ term and the cohomology of the associated single
complex $K\bu$ of $\bigoplus K^{p,q}$ is
\[
H_D^k(K\bu) \simeq H^k(X, \calf).
\]

In the augmented complex \eqref{e:cg},
by the construction of Godement's canonical resolution, 
the
\v{C}ech complex $\check{C}\bu(\fraku,\calf)$ injects into the complex $K\bu$
via a cochain map
\[
\epsilon\colon \check{C}^k(\fraku,\calf) \to K^{k,0} \hookrightarrow K^k,
\]
which gives rise to an induced map
\begin{equation} \label{e:induced}
\epsilon^*\colon \check{H}^k(\fraku,\calf) \to H_D^k(K\bu) = H^k(X,\calf)
\end{equation}
in cohomology.

\begin{defi}
A sheaf $\calf$ of abelian groups on a topological space $X$
is \term{acyclic on an open cover}\index{acyclic sheaf!on an open cover}
$\fraku = \{ U_{\ga} \}$ of $X$ if the cohomology
\[
H^k(U_{\ga_0 \cdots \ga_p}, \calf) =0
\]
for all $k > 0$ and all finite intersections
$U_{\ga_0 \cdots \ga_p}$ of open sets in $\fraku$.
\end{defi}

\begin{theorem} \label{t:csisom}
If a sheaf $\calf$ of abelian groups is acyclic on an open cover $\fraku = \{ U_{\ga}\}$
of a topological space $X$, then the induced map $\epsilon^*\colon
\check{H}^k(\fraku, \calf) \to H^k(X,\calf)$ is an isomorphism.\index{comparison!of \v{C}ech cohomology and
sheaf\\ cohomology}
\end{theorem}

\pf
Because $\calf$ is acyclic on each intersection
$U_{\ga_0 \cdots \ga_p}$, the cohomology of the $p$th column of \eqref{e:cg}
is $\prod H^0(U_{\ga_0 \cdots \ga_p}, \calf) =
\prod \calf(U_{\ga_0 \cdots \ga_p})$, so that the $E_1$ term
of the usual spectral sequence is
\begin{center}
\begin{pspicture}(-1,-.5)(9,2)
\psline{->}(0,-.1)(9,-.1)
\psline{->}(0,-.1)(0,1.7)
\psline(2,-.1)(2,1.7)
\psline(5,-.1)(5,1.7)
\psline(8,-.1)(8,1.7)
\uput{.2}[180](-0,.75){$E_1= H_d =$}
\rput(1,-0.4){$0$}
\rput(3.5,-0.4){$1$}
\rput(6.5,-0.4){$2$}
\rput(1,0.25){$\prod \calf(U_{\ga_0})$}
\rput(3.5,0.25){$\prod \calf(U_{\ga_0 \ga_1})$}
\rput(6.5,0.25){$\prod \calf(U_{\ga_0 \ga_1\ga_2})$}
\rput(1,.75){$0$}
\rput(3.5,.75){$0$}
\rput(6.5,.75){$0$}
\rput(1,1.25){$0$}
\rput(3.5,1.25){$0$}
\rput(6.5,1.25){$0$}
\uput{.2}[270](9,-.1){$p$}
\uput{.2}[180](0,1.7){$q$}
\rput(9.2,-.1){,}
\end{pspicture}
\end{center}
and the $E_2$ term is
\begin{center}
\begin{pspicture}(-1,-.8)(9,2)
\psline{->}(0,-.1)(9,-.1)
\psline{->}(0,-.1)(0,1.7)
\psline(2,-.1)(2,1.7)
\psline(5,-.1)(5,1.7)
\psline(8,-.1)(8,1.7)
\uput{.2}[180](-0,.75){$E_2= H_{\delta}H_d =$}
\rput(1,-0.4){$0$}
\rput(3.5,-0.4){$1$}
\rput(6.5,-0.4){$2$}
\rput(1,0.25){$\check{H}^0(\fraku, \calf)$}
\rput(3.5,0.25){$\check{H}^1(\fraku, \calf)$}
\rput(6.5,0.25){$\check{H}^2(\fraku, \calf)$}
\rput(1,.75){$0$}
\rput(3.5,.75){$0$}
\rput(6.5,.75){$0$}
\rput(1,1.25){$0$}
\rput(3.5,1.25){$0$}
\rput(6.5,1.25){$0$}
\uput{.2}[270](9,-.1){$p$}
\uput{.2}[180](0,1.7){$q$}
\rput(9.2,-.1){.}
\end{pspicture}
\end{center}
Hence, the spectral sequence degenerates at the $E_2$ term and there is an
isomorphism
\[
\epsilon^*\colon \check{H}^k(\fraku,\calf) \simeq H_D^k(K\bu) \simeq
H^k(X, \calf). 
\]
\epf

\begin{remark}
Although we used a spectral sequence argument to prove
Theorem~\ref{t:csisom}, in the proof there is no problem with
the extension of groups in the $E_{\infty}$ term,
since along each antidiagonal $\bigoplus_{p+q =k} E_{\infty}^{p,q}$
there is only one nonzero box.
For this reason, Theorem~\ref{t:csisom} holds for
sheaves of abelian groups, not just for sheaves of
vector spaces.
\end{remark}

\section{\v{C}ech Cohomology of a Complex of Sheaves}

Just as the cohomology of a sheaf can be computed using a \v{C}ech complex on
an open cover (Theorem~\ref{t:csisom}), the hypercohomology of a complex of sheaves can
also be computed using the \v{C}ech method.

Let $(\call\bu,d_{\call})$ be a complex of sheaves on a topological space $X$,
and $\fraku= \{ U_{\ga} \}$ an open cover of $X$.
To define the \v{C}ech cohomology of $\call\bu$ on $\fraku$,
let $K = \bigoplus K^{p,q}$ be the double complex
\[
K^{p,q} = \check{C}^p(\fraku, \call^q)
\]
with its two commuting differentials $\delta$ and $d_{\call}$.
We will call $K$ the \term{\v{C}ech--sheaf double complex}.\index{Cech-sheaf double complex@\v{C}ech-sheaf
double complex}
The \term{\v{C}ech cohomology} $\check{H}^*(\fraku,\call\bu)$
of $\call\bu$ is defined to be the cohomology of the single complex
\[
K\bu = \bigoplus K^k, \text{ where } K^k = \bigoplus_{p+q=k}
\check{C}^p(\fraku, \call^q) \text{ and } d_K=\delta+(-1)^p d_{\call},
\]
associated to the \v{C}ech--sheaf double complex.

\subsection{The Relation Between \v{C}ech Cohomology and Hypercohomology}

There is an analogue of Theorem~\ref{t:csisom} that allows us to
compute hypercohomology using an open cover.

\begin{theorem} \label{t:chisom}
If $\call\bu$ is a complex of sheaves of abelian groups
 on a topological space $X$ such that each sheaf $\call^q$ is
acyclic on the open cover $\fraku = \{ U_{\ga} \}$ of $X$, then there
is an isomorphism $\check{H}^k(\fraku, \call\bu) \simeq
\hh^k(X, \call\bu)$ between the \v{C}ech cohomology of $\call\bu$ on
the open cover $\fraku$ and the hypercohomology of $\call\bu$ on $X$.\index{comparison!of \v{C}ech cohomology and hypercohomology}
\end{theorem}

The \v{C}ech cohomology of the complex $\call^{\bullet}$ is the cohomology
of the associated single complex of the double complex
$\bigoplus_{p,q} \check{C}^p(\fraku,\call^q) = \bigoplus_{p,q}
\prod_{\ga} \call^q(U_{\ga_0 \cdots \ga_p})$,
where $\ga = (\ga_0 < \cdots < \ga_p)$.
The hypercohomology of the complex $\call^{\bullet}$ is the cohomology
of the associated single complex of the double complex
$\bigoplus_{q,r} \calc^r \call^q(X)$.
To compare the two, we form the triple complex with terms
\[
N^{p,q,r} =  \check{C}^p(\fraku, \calc^r\call^q)
\]
and three commuting differentials: the \v{C}ech differential $\delta_{\check{C}}$,
the differential $d_{\call}$ of the complex $\call\bu$, and
the Godement differential $d_{\calc}$.

Let $N^{\bullet,\bullet,\bullet}$ be any triple complex\index{triple complex} with three
commuting differentials $d_1$, $d_2$, and $d_3$ of degrees
$(1,0,0)$, $(0,1,0)$, and $(0,0,1)$, respectively.
Summing $N^{p,q,r}$ over $p$ and $q$, or over $q$ and $r$,
one can form two double complexes from $N^{\bullet,\bullet,\bullet}$:
\[
N^{k,r} = \bigoplus_{p+q=k} N^{p,q,r}
\]
with differentials
\[
\delta = d_1 + (-1)^p d_2, \quad d = d_3,
\]
and
\[
N'^{p,\ell} = \bigoplus_{q+r = \ell} N^{p,q,r}
\]
with differentials
\[
\delta' = d_1, \quad d' = d_2 + (-1)^q d_3.
\]

\begin{prop} \label{8p:triple}
For any triple complex $N^{\bullet,\bullet,\bullet}$,
the two associated double complexes $N^{\bullet, \bullet}$
and ${N'}^{\bullet,\bullet}$ have the same associated
single complex.
\end{prop}

\pf
Clearly, the groups
\[
N^n = \bigoplus_{k+r=n} N^{k,r} = \bigoplus_{p+q+r=n} N^{p,q,r}
\]
and
\[
N'^n = \bigoplus_{p+\ell=n} {N'}^{p,\ell} = \bigoplus_{p+q+r=n} N^{p,q,r}
\]
are equal.
The differential $D$ for $N\bu = \bigoplus_n  N^{n}$ is
\[
D = \delta + (-1)^k d = d_1 + (-1)^p d_2 + (-1)^{p+q} d_3.
\]
The differential $D'$ for ${N'}\bu = \bigoplus_n  {N'}^{n}$ is
\[
D' = \delta' + (-1)^p d' = d_1 + (-1)^p \big( d_2 + (-1)^{q} d_3\big) = D.
\]
\epf

Thus, any triple complex $N^{\bullet,\bullet,\bullet}$ has an
associated single complex $N^{\bullet}$ whose cohomology can
be computed in two ways, either from the double complex $(N^{\bullet,\bullet}, D)$
or from the double complex $(N'^{\bullet,\bullet}, D')$.

We now apply this observation to the \v{C}ech--Godement--sheaf triple complex
$$N^{\bullet,\bullet,\bullet} = \bigoplus \check{C}^p(\fraku, \calc^r\call^q)$$ of
the complex $\call\bu$ of sheaves.
The $k$th column of the double complex $N^{\bullet,\bullet} = \bigoplus N^{k,r}$
is
\[
\begin{xy}
(0,30)*+{\bigoplus_{p+q=k}\prod_{\ga_0 < \cdots < \ga_p} \calc^{r+1} \call^q (U_{\ga_0 \cdots \ga_p})}="c";
(0,20)*+{\bigoplus_{p+q=k}\prod_{\ga_0 < \cdots < \ga_p} \calc^r \call^q (U_{\ga_0 \cdots \ga_p})}="b";
(0,-2)*+{\bigoplus_{p+q=k}\prod_{\ga_0 < \cdots < \ga_p} \calc^0 \call^q (U_{\ga_0 \cdots \ga_p}),}="a";
(15,10)*+{\vdots};
{\ar (15,20)*+++{}; (15,30)*+++{}}
{\ar (15,10)*+++{}; (15,20)*+++{}}
{\ar (14,-2)*+++{}; (14,8)*+++{}}
\end{xy}
\]
where the vertical differential $d$ is the Godement differential $d_{\calc}$.
Since $\call\bu$ is acyclic on the open cover $\fraku = \{ U_{\ga} \}$,
this column is exact except in the zeroth row, and the zeroth row of the
cohomology $H_d$ is
\[
\bigoplus_{p+q=k} \prod_{\ga_0 < \cdots < \ga_p} \call^q(U_{\ga_0 \cdots \ga_p})
=
\bigoplus_{p+q=k} \check{C}^p(\fraku, \call^q) = \bigoplus_{p+q=k} K^{p,q}
= K^k,
\]
the associated single complex of the \v{C}ech--sheaf double complex.
Thus, the $E_1$ term of the first spectral sequence of $N^{\bullet,\bullet}$ is
\begin{center}
\begin{pspicture}(-1,-.5)(7,2)
\psline{->}(0,-.1)(7,-.1)
\psline{->}(0,-.1)(0,1.7)
\psline(2,-.1)(2,1.7)
\psline(4,-.1)(4,1.7)
\psline(6,-.1)(6,1.7)
\uput{.2}[180](-0,.75){$E_1 = H_d =$}
\rput(1,-0.4){$0$}
\rput(3,-0.4){$1$}
\rput(5,-0.4){$2$}
\rput(1,0.25){$K^0$}
\rput(3,0.25){$K^1$}
\rput(5,0.25){$K^2$}
\rput(1,.75){$0$}
\rput(3,.75){$0$}
\rput(5,.75){$0$}
\rput(1,1.25){$0$}
\rput(3,1.25){$0$}
\rput(5,1.25){$0$}
\uput{.2}[270](7,-.1){$k$}
\uput{.2}[180](0,1.7){$r$}
\rput(7.2,-.1){,}
\end{pspicture}
\end{center}
and so the $E_2$ term is
\[
E_2 = H_{\delta}(H_d) = H_{d_K}^*(K\bu) =
\check{H}^*(\fraku,\call\bu).
\]
Although we are working with abelian groups, there are no extension issues,
because each antidiagonal in $E_{\infty}$ contains only one nonzero group.
Thus, the $E_{\infty}$ term is
\begin{equation} \label{8e:isom2}
H_D^*(N\bu) \simeq E_2 = \check{H}^*(\fraku, \call\bu).
\end{equation}

On the other hand, the $\ell$th row of ${N'}^{\bullet,\bullet}$ is
\[
0 \to \bigoplus_{q+r=\ell} \check{C}^0(\fraku,\calc^r\call^q) \to \cdots \to
\bigoplus_{q+r=\ell} \check{C}^p(\fraku,\calc^r\call^q) \to
\bigoplus_{q+r=\ell} \check{C}^{p+1}(\fraku,\calc^r\call^q) \to \cdots ,
\]
which is the \v{C}ech cochain complex of the flasque sheaf
$\bigoplus_{q+r=\ell} \calc^r\call^q$ with differential $\delta'= \delta_{\check{\calc}}$.
Thus, each row of ${N'}^{\bullet,\bullet}$ is exact except in the
zeroth column, and the kernel of ${N'}^{0,\ell} \to {N'}^{1,\ell}$
is $M^{\ell} = \bigoplus_{q+r=\ell}\calc^r\call^q(X)$.
Hence, the $E_1$ term of the second spectral sequence is
\begin{center}
\begin{pspicture}(-1,-.5)(5,2)
\psline{->}(-.4,0)(5,0)
\psline{->}(-.4,0)(-.4,1.7)
\psline(1,0)(1,1.7)
\psline(2,0)(2,1.7)
\psline(3,0)(3,1.7)
\psline(4,0)(4,1.7)
\uput{.2}[180](-.4,.75){$E_1= H_{\delta'} =$}
\rput(0.3,-0.3){$0$}
\rput(1.5,-0.3){$1$}
\rput(2.5,-0.3){$2$}
\rput(3.5,-0.3){$3$}
\rput(0.3,0.25){$M^0$}
\rput(1.5,0.25){$0$}
\rput(2.5,0.25){$0$}
\rput(3.5,0.25){$0$}
\rput(0.3,.75){$M^1$}
\rput(1.5,.75){$0$}
\rput(2.5,.75){$0$}
\rput(3.5,.75){$0$}
\rput(0.3,1.25){$M^2$}
\rput(1.5,1.25){$0$}
\rput(2.5,1.25){$0$}
\rput(3.5,1.25){$0$}
\uput{.2}[270](5,0){$p$}
\uput{.2}[180](-.4,1.7){$\ell$}
\rput(5.2,0){.}
\end{pspicture}
\end{center}
The $E_2$ term is
\[
E_2 = H_{d'}(H_{\delta'}) = H_{d_M}^*(M\bu) = \hh^*(X,\call\bu).
\]
Since this spectral sequence for $N^{\bullet,\bullet}$ degenerates at
the $E_2$ term and converges to $H_{D'}^*({N'}\bu)$, there is an isomorphism
\begin{equation} \label{8e:isom1}
E_{\infty} = H_{D'}^*({N'}\bu) \simeq E_2  = \hh^*(X, \call\bu).
\end{equation}

By Proposition~\ref{8p:triple}, the two groups in \eqref{8e:isom2} and \eqref{8e:isom1}
are isomorphic.  In this way, one obtains an isomorphism
between the \v{C}ech cohomology and the hypercohomology of the complex $\call\bu$:
\[
\check{H}^*(\fraku,\call\bu) \simeq \hh^*(X,\call\bu). 
\]

\section{Reduction to the Affine Case}

Grothendieck proved his general algebraic de Rham theorem by reducing
it to the special case of an affine variety.
This section is an exposition of his ideas in \cite{grothendieck66}.

\begin{theorem}[Algebraic de Rham theorem]\label{t:algebraic}\index{algebraic de Rham theorem}
Let $X$ be a smooth algebraic variety defined over the complex
numbers,
and $X\suban$ its underlying complex manifold.
Then the singular cohomology of $X\suban$ with $\C$ coefficients can be
computed as the hypercohomology of the complex
$\Omega\alg\bu$ of sheaves of algebraic differential forms on $X$
with its Zariski topology:
\[
H^k(X\suban, \C) \simeq \hh^k(X, \Omega\alg\bu).
\]
\end{theorem}

By the isomorphism $H^k(X\suban,\C) \simeq \hh^k(X\suban,
\Omega\suban\bu)$ of the analytic de Rham theorem,
Grothendieck's algebraic de Rham theorem is equivalent to an
isomorphism in hypercohomology
\[
\hh^k(X, \Omega\alg\bu) \simeq \hh^k(X\suban, \Omega\suban\bu).
\]
The special case of Grothendieck's theorem for an affine variety is
especially interesting, since it does not involve hypercohomology.

\begin{cor}[The affine case]\label{c:affine}\index{algebraic de Rham theorem!the affine case}
Let $X$ be a smooth affine variety defined over the complex numbers
and $\big(\Omega\alg\bu(X)$, $d\big)$ the complex of algebraic differential
forms on $X$.
Then the singular cohomology with $\C$ coefficients of its underlying
complex manifold $X\suban$ can be computed as the cohomology
of its complex of
algebraic differential forms:
\[
H^k(X\suban, \C) \simeq h^k\big(\Omega\alg\bu(X)\big).
\]
\end{cor}

It is important to note that the left-hand side is the singular
cohomology of the complex manifold $X\suban$, not of the affine variety $X$.
In fact, in the Zariski topology, a constant sheaf on an irreducible
variety is always flasque (Example~\ref{exam:irreducible}), and hence
acyclic (Corollary~\ref{c:flasque_acyclic}), so that $H^k(X, \C) = 0$ for all $k > 0$ if $X$ is irreducible.

\subsection{Proof that the General Case Implies the Affine Case}

Assume Theorem~\ref{t:algebraic}.
It suffices to prove that for a smooth affine complex variety $X$, the
hypercohomology $\hh^k(X, \Omega\alg\bu)$ reduces to the cohomology
of the complex $\Omega\alg\bu(X)$.
Since $\Omega\alg^q$ is a coherent algebraic sheaf,
by Serre's vanishing theorem for an affine variety (Theorem~\ref{t:affinity}),
$\Omega\alg^q$ is acyclic on $X$.
%
%
%
By Theorem~\ref{t:acyclic},
\[
\hh^k(X, \Omega\alg\bu) \simeq
h^k\big(\Omega\alg\bu(X)\big). 
\]

\subsection{Proof that the Affine Case Implies the General Case}

Assume Corollary~\ref{c:affine}. The proof is based on the facts that every algebraic variety $X$
has an \term{affine open cover},
an open cover $\fraku = \{ U_{\ga} \}$ in which every
$U_{\ga}$ is an affine open set, and that the intersection of
two affine open sets is affine open.
The existence of an affine open cover for an algebraic variety
follows from the elementary fact that every quasi-projective variety
has an affine open cover;
since an algebraic variety by definition has an open cover
by quasi-projective varieties, it necessarily has an open cover
by affine varieties.

Since $\Omega\alg\bu$ is a complex of locally free
and hence coherent
algebraic sheaves, by Serre's vanishing theorem for an affine variety (Theorem~\ref{t:affinity}),
$\Omega\alg\bu$ is acyclic on an
affine open cover.
By Theorem~\ref{t:chisom}, there is an isomorphism
\begin{equation}\label{e:algisom}
\check{H}^*(\fraku, \Omega\alg\bu) \simeq \hh^*(X, \Omega\alg\bu)
\end{equation}
between the \v{C}ech cohomology of $\Omega\alg\bu$ on the
affine open cover $\fraku$
and the hypercohomology of $\Omega\alg\bu$ on $X$.
Similarly, by Cartan's theorem~B (because a complex
affine variety with the complex topology is Stein)
and Theorem~\ref{t:chisom}, the
corresponding statement in the analytic category is also true:
if $\fraku\suban := \{ (U_{\ga})_{\an}\}$, then
\begin{equation} \label{e:anisom}
\check{H}^*(\fraku\suban, \Omega\suban\bu) \simeq \hh^*(X\suban, \Omega\suban\bu).
\end{equation}

The \v{C}ech cohomology $\check{H}^*(\fraku,\Omega\alg\bu)$ is the
cohomology of the single complex associated to the double complex
$\bigoplus K\alg^{p,q} = \bigoplus \check{C}^p(\fraku, \Omega\alg^q)$.
The $E_1$ term of the usual spectral sequence of this double complex is
\begin{alignat*}{2}
E_{1,{\alg}}^{p,q} &= H_d^{p,q} = h_d^q(K^{p,\bullet})
= h_d^q\big( \check{C}^p(\fraku, \Omega\alg\bu)\big) \\
&= h_d^q\Big( \prod_{\ga_0 < \cdots < \ga_p} \Omega\alg\bu(U_{\ga_0
  \cdots \ga_p}) \Big)\\
&=\prod_{\ga_0 < \cdots < \ga_p}  h_d^q\big( \Omega\alg\bu(
U_{\ga_0 \cdots \ga_p}) \big) \\
&=\prod_{\ga_0 < \cdots < \ga_p}  H^q(U_{\ga_0 \cdots \ga_p,{\an}},
{\C}) &\quad&\text{(by Corollary~\ref{c:affine})}.
\end{alignat*}

A completely similar computation applies to the usual spectral sequence of
the double complex $\bigoplus K\suban^{p,q}$ $= \bigoplus_{p,q}
\check{C}^p(\fraku\suban, \Omega\suban^q)$ converging to the \v{C}ech
cohomology $\check{H}^*(\fraku\suban, \Omega\suban\bu)$:
the $E_1$ term of this spectral sequence is
\begin{alignat*}{2}
E_{1,{\an}}^{p,q} & = \prod_{\ga_0 < \cdots < \ga_p} h_d^q
\big( \Omega\suban\bu(U_{\ga_0 \cdots \ga_p,{\an}}) \big) \\
&=\prod_{\ga_0 < \cdots < \ga_p}  H^q(U_{\ga_0 \cdots \ga_p,{\an}},
\C) &\quad&\text{(by Corollary~\ref{c:analytic})}.
\end{alignat*}

The isomorphism in $E_1$ terms,
\[
E_{1,{\alg}} \overset{\sim}{\to} E_{1,{\an}},
\]
commutes with the \v{C}ech differential $d_1 = \delta$ and
induces an isomorphism in $E_{\infty}$ terms,
\[
\xymatrix@R=8pt{
E_{\infty,{\alg}} \ar@{=}[d] \ar[r]^-{\sim} & E_{\infty,{\an}}
\ar@{=}[d] \\
\check{H}^*(\fraku, \Omega\alg\bu) & \check{H}^*(\fraku\suban, \Omega\suban\bu).
}
\]
Combined with \eqref{e:algisom} and \eqref{e:anisom}, this gives
\[
\hh^*(X, \Omega\alg\bu) \simeq \hh^*(X\suban, \Omega\suban\bu),
\]
which, as we have seen, is equivalent to the algebraic de Rham theorem (Theorem~\ref{t:algebraic})
for a smooth complex algebraic variety.

\section{The Algebraic de Rham Theorem for an Affine Variety}

It remains to prove the algebraic de Rham theorem in the form of
Corollary~\ref{c:affine} for a smooth affine complex variety $X$.
This is the most difficult case and is in fact the heart of the
matter.
We give a proof that is different from Grothendieck's in \cite{grothendieck66}.

A \term{normal crossing divisor}\index{normal crossing divisor} on a smooth algebraic variety is a
divisor that is locally the zero set of an equation of the form
$z_1\cdots z_k =0$, where $z_1, \ldots, z_N$ are local parameters.
We first describe a standard procedure by which any smooth affine variety $X$
may be assumed to be the complement of a normal crossing divisor $D$
in a smooth complex projective variety $Y$.
Let $\bar{X}$ be the projective closure of $X$; for example,
if $X$ is defined by polynomial equations
\[
f_i(z_1, \ldots, z_N) = 0
\]
in $\C^N$, then $\bar{X}$ is defined by the equations
\[
f_i\Big( \frac{Z_1}{Z_0}, \ldots, \frac{Z_N}{Z_0}\Big) = 0
\]
in $\C P^N$, where $Z_0, \ldots, Z_N$ are the homogeneous coordinates
on $\C P^N$ and $z_i = Z_i/Z_0$.
In general, $\bar{X}$ will be a singular projective variety.
By Hironaka's resolution of singularities, there is a surjective regular map
$\pi\colon Y \to \bar{X}$ from a smooth projective variety $Y$
to $\bar{X}$ such that $\pi^{-1} (\bar{X} - X)$ is a normal crossing divisor $D$ in $Y$
and $\pi|_{Y-D}\colon Y-D \to X$ is an isomorphism.
Thus, we may assume that $X = Y - D$, with an inclusion map
$j\colon X \hookrightarrow Y$.

Let $\Omega_{Y\suban}^k(*D)$ be the sheaf of meromorphic $k$-forms on
$Y\suban$ that are holomorphic on $X\suban$ with poles of any order $\ge 0$ along $D\suban$
(order $0$ means no poles)
and let $\cala_{X\suban}^k$ be the sheaf of $\cinf$ complex-valued
$k$-forms on $X\suban$.
By abuse of notation, we use $j$ also to denote the inclusion
$X\suban \hookrightarrow Y\suban$.
The \term{direct image sheaf}\index{direct image!sheaf} $j_* \cala_{X\suban}^k$ is by definition
the sheaf on $Y\suban$ defined by
\[
\big(j_* \cala_{X\suban}^k\big)(V) =  \cala_{X\suban}^k (V\cap X\suban)
\]
for any open set $V \subset Y\suban$.
Since a section of $\Omega_{Y\suban}^k(*D)$ over $V$
is holomorphic on $V \cap X\suban$ and therefore smooth
there, the sheaf $\Omega_{Y\suban}^k(*D)$ of meromorphic
forms is a subsheaf of the sheaf $j_* \cala_{X\suban}^k$ of smooth forms.
The main lemma of our proof, due to Hodge and Atiyah
\cite[Lem.~17, p.~77]{hodge--atiyah}, asserts that the inclusion
\begin{equation} \label{e:meromorphic}
\Omega_{Y\suban}\bu(*D) \hookrightarrow j_* \cala_{X\suban}\bu
\end{equation}
of complexes of sheaves is a quasi-isomorphism.
This lemma makes essential use of the fact that $D$ is a normal crossing
divisor.
Since the proof of the lemma is quite technical, in order not to
interrupt the flow of the exposition, we postpone it to the end of the
chapter.

By Theorem~\ref{t:qi}, the quasi-isomorphism \eqref{e:meromorphic}
induces an isomorphism
\begin{equation} \label{e:isomhyper}
\hh^k\big(Y\suban, \Omega_{Y\suban}\bu(*D)\big) \simeq
\hh^k(Y\suban, j_*\cala_{X\suban}\bu)
\end{equation}
in hypercohomology.
If we can show that the right-hand side is $H^k(X\suban,\C)$ and the
left-hand side is $h^k\big(\Omega\alg\bu(X)\big)$, the algebraic de Rham
theorem for the affine variety $X$ (Corollary~\ref{c:affine}),
$h^k\big( \Omega\alg\bu(X)\big) \simeq H^k(X\suban,\C)$, will follow.

\subsection{The Hypercohomology of the Direct Image of a Sheaf of Smooth Forms}\index{hypercohomology!of the direct image of a sheaf of smooth forms}

To deal with the right-hand side of \eqref{e:isomhyper}, we prove a
more general lemma valid on any complex manifold.

\begin{lemma}
Let $M$ be a complex manifold and $U$ an open submanifold,
with $j\colon U \hookrightarrow M$ the inclusion map.
Denote the sheaf of smooth $\C$-valued $k$-forms on $U$ by
$\cala_U^k$.
Then there is an isomorphism
\[
\hh^k(M, j_*\cala_U\bu) \simeq H^k (U, \C).
\]
\end{lemma}

\pf
Let $\cala^0$ be the sheaf of smooth $\C$-valued functions on the complex
manifold $M$.
For any open set $V \subset M$, there is an $\cala^0(V)$-module
structure on $(j_*\cala_U^k)(V) = \cala_U^k(U \cap V)$:
\begin{align*}
\cala^0(V) \times \cala_U^k(U \cap V) &\to \cala_U^k(U \cap V),\\
(f, \omega) &\mapsto f\cdot \omega.
\end{align*}
Hence, $j_*\cala_U^k$ is a sheaf of $\cala^0$-modules on $M$.
As such,  $j_*\cala_U^k$ is a fine sheaf on $M$
(Section~\ref{ss:fine}).

Since fine sheaves are acyclic, by Theorem~\ref{t:acyclic},
\begin{alignat*}{2}
\hh^k(M, j_*\cala_U\bu) &\simeq h^k\big( (j_*\cala_U\bu)(M) \big) \\
&= h^k \big( \cala_U\bu(U) \big) &\quad&\text{(definition of
  $j_*\cala_U\bu$)}\\
&= H^k(U,\C) &\quad&\text{(by the smooth de Rham theorem)}.
\end{alignat*}
\epf

Applying the lemma to $M= Y\suban$ and $U=X\suban$, we obtain
\[
\hh^k(Y\suban, j_*\cala_{X\suban}\bu) \simeq H^k(X\suban, \C).
\]
This takes care of the right-hand side of \eqref{e:isomhyper}.

\subsection{The Hypercohomology of Rational and Meromorphic Forms}

Throughout this subsection, the smooth complex affine variety $X$ is the
complement of a normal crossing divisor $D$ in a smooth complex projective
variety $Y$.
Let $\Omega_{Y\suban}^q(nD)$ be the sheaf of meromorphic $q$-forms on
$Y\suban$ that are holomorphic on $X\suban$ with poles of order $\le n$
along $D\suban$.
As before, $\Omega_{Y\suban}^q(*D)$ is the sheaf of meromorphic
$q$-forms on $Y\suban$ that are holomorphic on $X\suban$ with at most poles
(of any order) along $D$.
Similarly, $\Omega_Y^q(*D)$ and $\Omega_Y^q(nD)$ are
their algebraic counterparts, the sheaves of
rational $q$-forms on $Y$ that are regular on $X$ with poles along $D$ of
arbitrary order or order $\le n$ respectively.
Then
\[
\Omega_{Y\suban}^q (*D) = \varinjlim_n \Omega_{Y\suban}^q(nD) \quad
\text{and} \quad
\Omega_{Y}^q (*D) = \varinjlim_n \Omega_{Y}^q(nD).
\]

Let $\Omega_X^q$ and $\Omega_Y^q$ be the sheaves of
regular $q$-forms on $X$ and $Y$, respectively; they are
what would be written $\Omega\alg^q$ if there is only one variety.
Similarly, let $\Omega_{X\suban}^q$ and $\Omega_{Y\suban}^q$ be sheaves
of holomorphic $q$-forms on $X\suban$ and $Y\suban$, respectively.
There is another description of the sheaf $\Omega_Y^q(*D)$
that will prove useful.
Since a regular form on $X = Y - D$ that is not defined on $D$
can have at most poles along $D$ (no essential singularities),
if $j\colon X \to Y$ is the inclusion map, then
\[
j_* \Omega_X^q = \Omega_Y^q(*D).
\]
Note that the corresponding statement in the analytic category is not
true:
if $j\colon X\suban \to Y\suban$ now denotes the inclusion of the
corresponding analytic manifolds, then in general
\[
j_* \Omega_{X\suban}^q \ne \Omega_{Y\suban}^q(*D)
\]
because a holomorphic form on $X\suban$ that is not defined
along $D\suban$ may have an essential singularity on $D\suban$.

Our goal now is to prove that the hypercohomology
$\hh^*\big(Y\suban, \Omega_{Y\suban}\bu(*D)\big)$ of the complex
$\Omega_{Y\suban}\bu(*D)$ of sheaves of meromorphic forms on $Y\suban$
is computable from the algebraic de Rham complex on $X$:
\[
\hh^k\big( Y\suban, \Omega_{Y\suban}^{\bullet}(*D)\big) \simeq
h^k\big( \Gamma(X, \Omega\alg\bu)\big).
\]
This will be accomplished through a series of isomorphisms.

First, we prove something akin to a GAGA principle for
hypercohomology.
The proof requires commuting direct limits and cohomology,
for which we shall invoke the following criterion.
A topological space is said to be \term{noetherian}\index{noetherian topological space}
if it satisfies the descending chain condition for closed sets:
any descending chain $Y_1 \supset Y_2 \supset \cdots$ of closed
sets must terminate after finitely many steps.
As shown in a first course in algebraic geometry, affine and projective varieties are noetherian \cite[Exa.~1.4.7, p.~5; Exer.1.7(b), p.~8; Exer.~2.5(a),
p.~11]{hartshorne}.

\begin{prop}[Commutativity of direct limit with cohomology]
 \label{p:limit}
Let $(\calf_{\ga})$ be a direct system of sheaves on a topological
space $Z$. The natural map\index{commutativity!of direct limit with cohomology}
\[
\varinjlim H^k(Z, \calf_{\ga}) \to H^k(Z, \varinjlim \calf_{\ga})
\]
is an isomorphism if
\begin{enumerate}[(i)]
\item\label{ch2:prp2102.i}  $Z$ is compact; or
\item\label{ch2:prp2102.ii}  $Z$ is noetherian.
\end{enumerate}
\end{prop}

\pf
For \ref{ch2:prp2102.i}, see \cite[Lem.~4, p.~61]{hodge--atiyah}.
For \ref{ch2:prp2102.ii}, see \cite[Ch.~III, Prop.~2.9, p.~209]{hartshorne}
or \cite[Ch.~II, remark after Th.~4.12.1, p.~194]{godement}.
 \epf

\begin{prop} \label{p:gaga}
In the notation above, there is an isomorphism in hypercohomology
\[
\hh^*\big(Y, \Omega_Y\bu(*D)\big) \simeq
\hh^*\big(Y\suban, \Omega_{Y\suban}\bu(*D)\big).
\]
\end{prop}

\pf
Since $Y$ is a projective variety and each $\Omega_Y\bu(nD)$ is locally
free, we can apply Serre's GAGA principle \eqref{e:gaga} to get an isomorphism
\[
H^p\big(Y, \Omega_Y^q(nD)\big) \simeq
H^p\big(Y\suban, \Omega_{Y\suban}^q(nD)\big).
\]
Next, take the direct limit of both sides as $n \to \infty$.
Since the projective variety $Y$ is noetherian and the complex
manifold $Y\suban$ is compact, by Proposition~\ref{p:limit}, we obtain
\[
H^p\big(Y, \varinjlim_n \Omega_Y^q(nD)\big) \simeq
H^p\big( Y\suban, \varinjlim_n \Omega_{Y\suban}^q(nD)\big),
\]
which is
\[
H^p\big(Y, \Omega_Y^q(*D)\big)  \simeq H^p\big( Y\suban, \Omega_{Y\suban}^q(*D) \big).
\]

%

Now the two cohomology groups $H^p\big(Y, \Omega_Y^q(*D)\big)$
and $H^p\big( Y\suban, \Omega_{Y\suban}^q(*D) \big)$ are the $E_1$ terms of
the second spectral sequences of the hypercohomologies of
$\Omega_Y\bu(*D)$
and $\Omega_{Y\suban}\bu(*D)$, respectively (see \eqref{e:e1}).
An isomorphism of the $E_1$ terms induces an isomorphism of the
$E_{\infty}$ terms.
Hence,
\[
\hh^*\big(Y, \Omega_Y\bu(*D)\big) \simeq
\hh^*\big(Y\suban, \Omega_{Y\suban}\bu(*D)\big). 
\]
\epf

\begin{prop} \label{p:yx}
In the notation above, there is an isomorphism
\[
\hh^k\big(Y, \Omega_Y\bu(*D)\big) \simeq \hh^k(X, \Omega_X\bu)
\]
for all $k \ge 0$.
\end{prop}

\pf
If $V$ is an affine open set in $Y$, then $V$ is noetherian and so
by Proposition~\ref{p:limit}\ref{ch2:prp2102.ii}, for $p > 0$,
\begin{align*}
H^p\big(V, \Omega_Y^q(*D)\big)
&= H^p\big(V, \varinjlim_n \Omega_Y^q(nD)\big) \\
&\simeq \varinjlim_n H^p\big(V, \Omega_Y^q(nD) \big) \\
&= 0,
\end{align*}
the last equality following from Serre's vanishing theorem (Theorem~\ref{t:affinity}),
since $V$ is affine and $\Omega_Y^q(nD)$ is locally free and
therefore coherent.
Thus, the complex of sheaves $\Omega_Y\bu(*D)$ is
acyclic on any
affine open cover $\fraku = \{ U_{\ga} \}$ of $Y$.
By Theorem~\ref{t:chisom}, its hypercohomology can be
computed from its \v{C}ech cohomology:
\[
\hh^k\big(Y, \Omega_Y\bu(*D)\big) \simeq
\check{H}^k\big(\fraku, \Omega_Y\bu(*D) \big).
\]

Recall that if $j\colon X \to Y$ is the inclusion map, then
$\Omega_Y\bu(*D) = j_* \Omega_X\bu$.
By definition, the \v{C}ech cohomology
$\check{H}^k\big(\fraku, \Omega_Y\bu(*D) \big)$ is the cohomology
of the associated single complex of the double complex
\begin{align}
K^{p,q} &= \check{C}^p\big(\fraku, \Omega_Y^q(*D)\big) = \check{C}^p(\fraku,
j_*\Omega_X^q) \notag\\
&= \prod_{\ga_0 < \cdots < \ga_p} \Omega^q(U_{\ga_0 \cdots \ga_p} \cap
X). \label{e:y}
\end{align}

Next we compute the hypercohomology ${\hh}^k(X, \Omega_X\bu)$.
The restriction $\fraku|_X :=\{ U_{\ga} \cap X \}$ of $\fraku$ to $X$
is an affine open cover of $X$.
Since $\Omega_X^q$ is locally free \cite[Ch.~III, Th.~2, p.~200]{shafarevich}, by Serre's vanishing theorem for
an affine variety again,
\[
H^p(U_{\ga}\cap X, \Omega_X^q) = 0 \quad\text{for all } p > 0.
\]
Thus, the complex of sheaves $\Omega_X\bu$ is acyclic on the open cover $\fraku|_X$ of $X$.
By Theorem~\ref{t:chisom},
\[
\hh^k(X, \Omega_X\bu) \simeq \check{H}^k(\fraku|_X, \Omega_X\bu).
\]
The \v{C}ech cohomology $\check{H}^k(\fraku|_X, \Omega_X\bu)$ is the
cohomology of the single complex associated to the double complex
\begin{align}
K^{p,q} &= \check{C}^p (\fraku|_X, \Omega_X^q) \notag\\
&= \prod_{\ga_0 < \cdots < \ga_p} \Omega^q(U_{\ga_0 \cdots \ga_p} \cap
X). \label{e:x}
\end{align}
Comparing \eqref{e:y} and \eqref{e:x}, we get an isomorphism
\[
\hh^k\big(Y, \Omega_Y\bu(*D)\big) \simeq \hh^k(X, \Omega_X\bu)
\]
for every $k \ge 0$.
 \epf

Finally, because $\Omega_X^q$ is locally free, by Serre's vanishing
theorem
for an affine variety still again,
$H^p(X, \Omega_X^q) =0$ for all $p > 0$.
Thus, $\Omega_X\bu$ is a complex of acyclic sheaves on $X$.
By Theorem~\ref{t:acyclic}, the hypercohomology
$\hh^k(X,\Omega_X\bu)$ can be computed from the complex of global
sections of $\Omega_X\bu$:
\begin{equation} \label{e:alg}
\hh^k(X, \Omega_X\bu) \simeq h^k\big( \Gamma(X, \Omega_X\bu) \big)
= h^k \big(\Omega_{\alg}\bu(X)\big).
\end{equation}

Putting together Propositions \ref{p:gaga} and \ref{p:yx}
with \eqref{e:alg}, we get the desired interpretation
\[
\hh^k\big(Y\suban, \Omega_{Y\suban}\bu(*D)\big) \simeq
h^k\big( \Omega_{\alg}\bu(X) \big)
\]
of the left-hand side of \eqref{e:isomhyper}.
Together with the interpretation of the right-hand side of
\eqref{e:isomhyper} as $H^k(X\suban, \C)$,
this gives Grothendieck's algebraic de Rham theorem for an
affine variety,
\[
H^k(X\suban, \C) \simeq h^k\big( \Omega_{\alg}\bu(X) \big). 
\]

%
%

\subsection{Comparison of Meromorphic and Smooth Forms}\index{comparison!of meromorphic and smooth\\ forms}

It remains to prove that \eqref{e:meromorphic} is a quasi-isomorphism.  We will
reformulate the lemma in slightly more general terms.
Let $M$ be a complex manifold of complex dimension $n$,
let $D$ be a normal crossing divisor in $M$,
and let $U = M - D$ be the complement of $D$ in $M$, with $j\colon U
\hookrightarrow M$ the inclusion map.
Denote by $\Omega_M^q(*D)$ the sheaf of meromorphic $q$-forms on $M$ that
are holomorphic on $U$ with at most poles along $D$,
and by $\cala_U^q:=\cala_U^q(\ , \C)$ the sheaf of smooth $\C$-valued $q$-forms on $U$.
For each $q$, the sheaf $\Omega_M^q(*D)$ is a subsheaf of
$j_*\cala_U^q$.

\begin{lemma}[Fundamental lemma of Hodge and Atiyah ({\cite[Lem.~17, p.~77]{hodge--atiyah}})]\label{10l:ha}
The inclusion $\Omega_M\bu(*D) \hookrightarrow j_*\cala_U\bu$ of
complexes of sheaves is a quasi-isomorphism.\index{fundamental lemma!of Hodge and Atiyah}
\end{lemma}

\pf
We remark first that this is a \emph{local} statement.
Indeed, the main advantage of using sheaf theory is to reduce the global statement
of the algebraic de Rham theorem for an affine variety to a local result.
The inclusion $\Omega_M\bu(*D) \hookrightarrow j_*\cala_U\bu$
of complexes induces a morphism of cohomology sheaves
$\calh^*\big( \Omega_M\bu(*D)\big) \to \calh^*(j_*\cala_U\bu)$.
It is a general fact in sheaf theory that a morphism of sheaves is an isomorphism
if and only if its stalk maps are all isomorphisms \cite[Prop.~1.1, p.~63]{hartshorne}, so we will first examine the stalks of the sheaves in question.
There are two cases:  $p\in U$ and $p \in D$.
For simplicity, let $\Omega_p^q :=(\Omega_M^q)_p$ be the stalk of
$\Omega_M^q$ at $p \in M$ and let $\cala_p^q:= (\cala_U^q)_p$ be the
stalk of $\cala_U^q$ at $p \in U$.

\textbf{Case 1:}  At a point $p\in U$, the stalk of $\Omega_M^q(*D)$ is $\Omega_p^q$,
and the stalk of $j_*\cala_U^q$ is $\cala_p^q$.
Hence, the stalk maps of the inclusion $\Omega_M\bu(*D)
\hookrightarrow j_*\cala_U\bu$
at $p$ are
\begin{equation} \label{e:qi}
\bfig
\xymatrix{
0 \ar[r] & \Omega_p^0 \ar[r] \ar[d] & \Omega_p^1 \ar[r] \ar[d]
& \Omega_p^2 \ar[r] \ar[d]  & \cdots \\
0 \ar[r] & \cala_p^0 \ar[r] & \cala_p^1 \ar[r]
& \cala_p^2 \ar[r]  & \cdots .
}
\efig
\end{equation}
Being a chain map, \eqref{e:qi} induces a homomorphism in cohomology.
By the holomorphic Poincar\'e lemma (Theorem~\ref{t:holompoincare}), the cohomology of the top row of
\eqref{e:qi} is
\[
h^k(\Omega_p\bu) = \begin{cases}
\ \C &\text{for } k=0,\\
\ 0 &\text{for } k >0.
\end{cases}
\]
By the complex analogue of the smooth Poincar\'e
lemma (\cite[Sec.~4, p.~33]{bott--tu} and \cite[p.~38]{griffiths--harris}), the cohomology of the bottom row
of \eqref{e:qi} is
\[
h^k(\cala_p\bu) =\begin{cases}
\ \C &\text{for } k=0,\\
\ 0 &\text{for } k >0.
\end{cases}
\]
Since the inclusion map \eqref{e:qi} takes
$1\in \Omega_p^0$ to $1 \in \cala_p^0$,
it is a quasi-isomorphism.

By Proposition~\ref{p:stalkcohomology}, for $p \in U$,
\[
\calh^k\big( \Omega_M\bu(*D)\big)_p \simeq h^k\big(
(\Omega_M\bu(*D))_p\big) = h^k(\Omega_p\bu)
\]
and
\[
\calh^k(j_*\cala_U\bu)_p \simeq h^k\big( (j_*\cala_U\bu)_p \big) =
h^k(\cala_p\bu).
\]
Therefore, by the preceding paragraph, at $p\in U$
the inclusion $\Omega_M\bu(*D) \hookrightarrow j_*\cala_U\bu$ induces
an isomorphism of stalks
\begin{equation} \label{10e:stalkcomparison}
\calh^k\big( \Omega_M\bu(*D) \big)_p \simeq \calh^k(j_*\cala_U\bu)_p
\end{equation}
for all $k > 0$.

\textbf{Case 2:}  Similarly, we want to show that \eqref{10e:stalkcomparison} holds for
 $p \notin U$, i.e., for $p \in D$.
 Note that to show the stalks of these sheaves at $p$ are isomorphic,
 it is enough to show the spaces of sections are isomorphic over
 a neighborhood basis of polydisks.\index{polydisk}

Choose local coordinates $z_1, \ldots, z_n$ so that $p = (0, \ldots,
0)$
is the origin and $D$ is the zero set of $z_1 \cdots z_k =0$ on some
coordinate neighborhood of $p$.
Let $P$ be the polydisk
$P = \Delta^n := \Delta \times \cdots \times \Delta\ (n \text{ times})$, where
$\Delta$ is a small disk centered at the origin in $\C$,
say of radius $\epsilon$ for some $\epsilon > 0$.
Then $P \cap U$ is the \term{polycylinder}\index{polycylinder}
\begin{align*}
P^* &:= P \cap U = \Delta^n \cap (M-D) \\
&=\{ (z_1, \ldots, z_n) \in \Delta^n \mid z_i \ne 0 \text{ for } i= 1, \ldots,
k \}\\
&= (\Delta^*)^k \times \Delta^{n-k},
\end{align*}
where $\Delta^*$ is the punctured disk $\Delta - \{0\}$ in $\C$.
Note that $P^*$ has the homotopy type of the torus $(S^1)^k$.
For $1 \le i \le k$, let $\gamma_i$ be a circle wrapping once
around the $i$th $\Delta^*$.
Then a basis for the homology of $P^*$ is given by the submanifolds
$\prod_{i\in J} \gamma_i$ for all the various subsets $J \subset [1,k]$.

Since on the polydisk $P$,
\[
(j_*\cala_U\bu)(P) = \cala_U\bu(P \cap U) = \cala\bu(P^*),
\]
the cohomology of the complex $(j_*\cala_U\bu)(P)$ is
\begin{align}
h^*\big( (j_*\cala_U\bu)(P)\big) &= h^*\big( \cala\bu(P^*)\big)\notag\\
&= H^*(P^*, \C) \simeq H^*\big( (S^1)^k, \C \big) \notag\\
&=\bigwedge\left(\left[\frac{dz_1}{z_1}\right] , \ldots, \left[\frac{dz_k}{z_k}\right] \right), \label{10e:free}
\end{align}
the free exterior algebra on the $k$ generators $[dz_1/z_1], \ldots, [dz_k/z_k]$.
Up to a constant factor of $2\pi i$, this basis is dual to the homology basis cited above, as we can see
by integrating over products of loops.

For each $q$, the inclusion $\Omega_M^q(*D) \hookrightarrow j_*\cala_U^q$ of sheaves
induces an inclusion of groups of sections over a polydisk $P$:
\[
\Gamma\big( P, \Omega_M^q(*D) \big) \hookrightarrow \Gamma(P,
j_*\cala_U^q).
\]
As $q$ varies, the inclusion of complexes
\[
i\colon \Gamma\big( P, \Omega_M\bu(*D) \big) \rightarrow \Gamma(P,
j_*\cala_U\bu)
\]
induces a homomorphism in cohomology
\begin{equation}\label{e:onto}
i^*\colon h^*\big(\Gamma( P, \Omega_M\bu(*D) )\big)
\to h^*\big( \Gamma(P, j_*\cala_U\bu) \big)=
\bigwedge \left( \left[\frac{dz_1}{z_1}\right] , \ldots, \left[\frac{dz_k}{z_k}\right] \right).
\end{equation}
Since each $dz_j/z_j$ is a closed meromorphic form on $P$ with poles
along $D$, it defines a cohomology class in $h^*\big( \Gamma(P,
\Omega_M\bu(*D)) \big)$.
Therefore, the map $i^*$ is surjective.
If we could show $i^*$ were an isomorphism, then by taking the direct limit over all
polydisks $P$ containing $p$, we would obtain
\begin{equation}
\calh^*\big( \Omega_M\bu(*D) \big)_p
\simeq \calh^*(j_*\cala_U\bu)_p \quad\text{for } p \in D,
\end{equation}
which would complete the proof of the fundamental lemma (Lemma~\ref{10l:ha}).

We now compute the cohomology of the complex $\Gamma\big(P, \Omega_M\bu(*D)\big)$.

\begin{prop} \label{10p:generators}
Let $P$ be a polydisk $\Delta^n$ in $\C^n$, and $D$ the normal crossing divisor defined in $P$
by $z_1\cdots z_k =0$.
The cohomology ring $h^*\big( \Gamma(P, \Omega\bu(*D))\big)$ is generated by
$[dz_1/z_1], \ldots,$ $[dz_k/z_k]$.
\end{prop}

\begin{proof}
The proof is by induction on the number $k$ of irreducible components
of the singular set $D$.  When $k=0$, the divisor $D$ is empty and
meromorphic forms on $P$ with poles along $D$ are holomorphic.
By the holomorphic Poincar\'e lemma,
\[
h^*\big( \Gamma(P, \Omega\bu)\big) =  H^*(P, \C) = \C.
\]
This proves the base case of the induction.

The induction step is based on the following lemma.

\begin{lemma}\label{l:mero}
Let $P$ be a polydisk $\Delta^n$, and $D$ the normal crossing divisor
defined by $z_1 \cdots z_k =0$ in $P$.
Let $\varphi \in \Gamma\big( P, \Omega^q(*D)\big)$ be a closed meromorphic
$q$-form on $P$ that is holomorphic on $P^* := P -D$ with
at most poles along
$D$.
Then there exist closed meromorphic forms $\varphi_0 \in \Gamma\big(P, \Omega^q(*D) \big)$
and $\ga_1 \in \Gamma\big( P, \Omega^{q-1}(*D) \big)$, which have no poles along $z_1 =0$,
such that their cohomology classes satisfy the relation
\[
[\varphi] = [\varphi_0] + \left[ \frac{dz_1}{z_1}\right] \wedge [\ga_1].
\]
\end{lemma}

\pf
Our proof is an elaboration of the proof of Hodge and Atiyah
\cite[Lem.~17, p.~77]{hodge--atiyah}.
We can write $\varphi$ in the form
\[
\varphi = dz_1 \wedge \ga + \gb,
\]
where the meromorphic $(q-1)$-form $\ga$ and the $q$-form $\gb$ do not involve
$dz_1$.
Next, we expand $\ga$ and $\gb$ as Laurent series in $z_1$:
\begin{align*}
\ga &= \ga_0 + \ga_1 z_1^{-1} + \ga_2 z_1^{-2} + \cdots + \ga_r z_1^{-r},\\
\gb &= \gb_0 + \gb_1 z_1^{-1} + \gb_2 z_1^{-2} + \cdots + \gb_r
z_1^{-r},
\end{align*}
where $\ga_i$ and $\gb_i$ for $1 \le i \le r$ do not involve $z_1$ or
$dz_1$ and are meromorphic in the other variables,
and $\ga_0$, $\gb_0$ are holomorphic in $z_1$, are meromorphic in the
other variables, and do not involve $dz_1$.
Then
\[
\varphi = (dz_1 \wedge \ga_0 + \gb_0) +
\left(dz_1 \wedge \sum_{i=1}^r \ga_i z_1^{-i} + \sum_{i=1}^r \gb_i
z_1^{-i}\right).
\]

Set $\varphi_0 =dz_1 \wedge \ga_0 + \gb_0$.
By comparing the coefficients of $z_1^{-i}\,dz_1$ and $z_1^{-i}$, we
deduce from the condition $d\varphi =0$,
\begin{tabbing}
\hspace{1.4in} \= $d\ga_1$ \= $= d\ga_2 + \gb_1$ \= $= d\ga_3 + 2\gb_2$ \= $= \cdots$ \= $=r\gb_r$ \= $=0$,\\
\> $d\gb_1$ \> $= d\gb_2$            \> $= d\gb_3$      \> $= \cdots$ \> $=d\gb_r$ \> $=0$,
\end{tabbing}
and $d\varphi_0 =0$.

We can write
\begin{equation} \label{e:form}
\varphi = \varphi_0 + \frac{dz_1}{z_1} \wedge \ga_1 + \left(
dz_1 \wedge \sum_{i=2}^r \ga_i z_1^{-i} + \sum_{i=1}^r \gb_i z_1^{-i}
\right).
\end{equation}
It turns out that the term within the parentheses in \eqref{e:form}
is $d\theta$ for
\[
\theta = -\frac{\ga_2}{z_1} - \frac{\ga_3}{2z_1^2} - \cdots
-\frac{\ga_r}{(r-1)z_1^{r-1}}.
\]
In \eqref{e:form}, both $\varphi_0$ and $\ga_1$ are closed.
Hence, the cohomology classes satisfy the relation
\[
[\varphi] = [\varphi_0] + \left[\frac{dz_1}{z_1}\right] \wedge [\ga_1].
\]
\end{proof}

Since $\varphi_0$ and $\alpha_1$ are meromorphic forms which do
not have poles along $z_1 = 0$,
their singularity set is contained in the normal crossing divisor
defined by $z_2 \cdots z_k=0$, which has $k-1$ irreducible components.
By induction, the cohomology classes of $\varphi_0$ and $\ga_1$ are
generated by $[dz_2/z_2], \ldots, [dz_k/z_k]$.
Hence, $[\varphi]$ is a graded-commutative polynomial in $[dz_1/z_1], \ldots, [dz_k/z_k]$.
This completes the proof of Proposition~\ref{10p:generators}.
\end{proof}

%

\begin{prop}
Let $P$ be a polydisk $\Delta^n$ in $\C^n$, and $D$ the normal crossing divisor defined
by $z_1\cdots z_k =0$  in $P$.
Then there is a ring isomorphism
\[
h^*\big( \Gamma(P, \Omega\bu(*D) )\big) \simeq
\bigwedge\left(\left[\frac{dz_1}{z_1}\right] , \ldots, \left[\frac{dz_k}{z_k}\right] \right).
\]
\end{prop}

\begin{proof}
By Proposition~\ref{10p:generators}, the graded-commutative algebra $h^*\big(\Gamma(P, \Omega\bu(*D))\big)$ is generated by $[dz_1/z_1], \ldots, [dz_k/z_k]$.
It remains to show that these generators satisfy no algebraic relations other than those implied by graded commutativity.
Let $\omega_i = dz_i/z_i$ and $\omega_I:= \omega_{i_1 \cdots i_r} := \omega_{i_1} \wedge \cdots \wedge \omega_{i_r}$.
Any linear relation among the cohomology classes $[\omega_I]$ in $h^*\big( \Gamma(P, \Omega\bu(*D) )\big)$
would be, on the level of forms, of the form
\begin{equation} \label{10e:relation}
\sum c_I \omega_I = d\xi
\end{equation}
for some meromorphic form $\xi$ with at most poles along $D$.
But by restriction to $P-D$, this would give automatically a relation in $\Gamma(P, j_* \mathcal{A}_U^q)$.
Since $h^*\big(\Gamma( P, j_* \mathcal{A}_U^q)\big) = \bigwedge\left([\omega_1] , \ldots, [\omega_k] \right)$
is freely generated by $[\omega_1], \ldots, [\omega_k]$ (see \eqref{10e:free}),
the only possible relations \eqref{10e:relation} are all implied by graded commutativity.
\end{proof}

Since the inclusion
$\Omega_M\bu(*D) \hookrightarrow j_*\cala_U^*$
induces an isomorphism
\[
\calh^*\big(\Omega_M\bu(*D)\big)_p \simeq
\calh^*(j_*\cala_U\bu)_p
\]
of stalks of cohomology sheaves for all $p$,
the inclusion $\Omega_M\bu(*D) \hookrightarrow j_*\cala_U^*$ is
a quasi-isomorphism.
This completes the proof of Lemma~\ref{10l:ha}.
\end{proof}


\end{document}